\documentclass[]{article}
\usepackage{geometry}
\usepackage{algorithm}
\usepackage{algpseudocode}

\usepackage{amsmath}
\usepackage{amssymb}

\usepackage{bigfoot}

\usepackage{enumerate}

\usepackage{fancyhdr}
\usepackage{fancyvrb}

\usepackage[hidelinks]{hyperref} 

\usepackage{amsthm} 
\newtheorem{definition}{Definition}

\newtheorem{cor}{Corollary}
\newtheorem{prop}{Proposition}
\newtheorem{lemma}{Lemma}

\newtheorem{theorem}{Theorem}
\newtheorem*{theorem*}{Theorem}
\newtheorem{remark}{Remark}

\newtheorem*{example*}{Example}

\usepackage{mathrsfs}
\usepackage{physics}

\usepackage{xcolor} 

\usepackage{caption}
\usepackage{subcaption}
\usepackage{graphicx}

\title{\vspace{-1cm}Computing nonlinear Schr\"odinger equations with Hermite functions beyond harmonic traps}
\author{
Valeria Banica$^{1}$
\and
Georg Maierhofer$^{2}$
\and
Katharina Schratz$^{1}$
}
\date{}

\newcommand{\affil}{\vspace{0.5em}\noindent}

\usepackage{xcolor}
\newcommand\georg[1]{\textcolor{black}{#1}}
\newcommand\revisions[1]{{#1}}

\usepackage[toc,page]{appendix}

\usepackage{enumerate}

\usepackage{bm}

\usepackage{mathdots}

\newcommand{\R}{\mathbb{R}}
\newcommand{\N}{\mathbb{N}}

\usepackage{bibspacing}
\begin{document}
\maketitle
\begin{center}
\affil
$^{1}$ Laboratoire Jacques-Louis Lions (UMR 7598), Sorbonne Université, Paris, France

\affil
$^{2}$ Department of Applied Mathematics and Theoretical Physics, University of Cambridge, UK
\end{center}\vspace{0.1cm}
\begin{abstract}
Hermite basis functions are a \revisions{classical} tool for the spatial discretisation of Schr\"odinger equations with harmonic potential. In this work, \revisions{we prove that their favourable stability properties extend to Schr\"odinger equations without a trap: the free Schr\"odinger flow is stable in the weighted Sobolev spaces $\Sigma^k$ which govern the convergence of Hermite spectral methods. This makes the Hermite basis} a natural \revisions{discretisation} for a larger class of nonlinear \revisions{Schr\"odinger} equations posed on \revisions{the full space $\mathbb{R}^d$, avoiding artificial periodisation and the associated distortion of the dynamics incurred by domain truncation in Fourier methods.} \revisions{Within this framework we provide a rigorous fully discrete convergence analysis of a splitting method for the cubic nonlinear Schr\"odinger equation. In addition, by combining the Hermite basis with a gauge transform,} we introduce a novel, fully explicit, unconditionally stable numerical method for the derivative nonlinear Schr\"odinger equation. Our theoretical results are supported with numerical examples across various nonlinearities and dimensions\revisions{, which showcase the accuracy, stability and resulting efficiency of this Hermite basis approach}.
\end{abstract}

\paragraph{Keywords:} Hermite functions; spectral methods; nonlinear Schr\"odinger equations
\paragraph{MSC codes:} 65M70, 35Q41, 35Q55 
\section{Introduction}
In this work we are interested in developing numerical methods for nonlinear Schr\"odinger type equations on unbounded domains of the form
\begin{align}\label{eqn:general_schroedinger}
   \begin{cases} i\partial_t \psi=-\Delta\psi+V(\psi,x),\ &(t,x)\in [0,T]\times \mathbb{R}^d,\\
   \psi(0,x)=\psi_0(x),\ &x\in\mathbb{R}^d,
   \end{cases}
\end{align}
where $T$ is the time of existence of solutions to \eqref{eqn:general_schroedinger} (this may also be infinity) and $\psi_0$ denotes the initial value. We want to deal with the unboundedness of the spatial domain $\mathbb{R}^d$. In  PDE analysis  the full space setting is easier to handle than, for instance, the torus due to the dispersive nature of the solution on $\mathbb{R}^d$. Numerically, on the other hand, unbounded domains are computationally much more challenging \revisions{and often artificial boundary conditions have to be imposed}.

Schr\"odinger equations appear as central models in a number of domains such as solid state physics and fibre optics and, as a result, their numerical approximation has been widely studied \cite{Jin_Markowich_Sparber_2011}. In particular, some of the most popular approaches are finite difference schemes \cite{markowich1999numerical,faou2025fully} and, the focus of the present work, spectral methods \cite{gottlieb1977numerical,iserles2009first,trefethen2000}. So far, Fourier spectral methods \cite{bao2003numerical,bao2002time,bao2003anumerical,ostermann2022error,banica2024numerical} are, in practice, often the preferred computational tool, due to the diagonal structure of $\Delta$ in the Fourier basis lending itself very nicely for the construction of splitting schemes. The downside, however, is that the use of Fourier series requires periodic boundary conditions. Thus the common practice is to truncate the domain (usually to a symmetric interval $[-L,L]\subset\mathbb{R}$) and apply the Fourier spectral method on $[-L,L]$. A particularly promising way of adaptively choosing this truncation was recently presented in \cite{iserles2024convergence}. However, the dynamics of this ``truncated'' surrogate Schr\"odinger equation is fundamentally different from the original equation \eqref{eqn:general_schroedinger} and, in particular, the unbounded domain case is not recovered in the large-box limit as rigorously shown in \cite{faou2016weakly}. As a result, we cannot hope to get a reliable description of the dynamics of \eqref{eqn:general_schroedinger} by assuming a large torus $\mathbb{T}^d_L = [-L,L]^d$ instead of the full space $\mathbb{R}^d$.

This observation has informed a recent push in the use of spatial discretisations that are more adapted to unbounded domains \cite{iserles2021solving}. A possible choice of basis for this purpose are Malmquist--Takenaka functions discovered by \cite{takenaka1925orthogonal,malmquist1925determination} and recently reintroduced and studied more thoroughly as part of a larger class of orthogonal systems with skew-Hermitian differentiation matrices in \cite{luong2023approximation,iserles2023approximation,iserles2020family}.

\revisions{Motivated by this line of research, we address a natural question: Can Hermite functions (which are so successfully used in Gross--Pitaevskii equations) be used for Schr\"odinger equations without trapping potential?} To answer this, we revisit the Hermite basis, which traditionally has been used if $V$ in \eqref{eqn:general_schroedinger} contains a harmonic trap, i.e. $V(\psi,x)=|x|^2+\tilde{V}(\psi,x)$ (also referred to as Gross--Pitaevskii systems). This is because the Hermite functions are eigenfunctions of the operator $-\Delta+|x|^2$ thus permitting the construction of highly efficient splitting methods. The properties of this basis (in combination with time splitting methods) in the  Gross--Pitaevskii case have been successfully studied \cite{Gauckler2011,Thalhammer2012,ThCaNe2009,Lasser_Lubich_2020,carles2025time} and applied also in the compactified-time case \cite{carles2025scattering}.

In the context of splitting methods, a central point of concern for the application of Hermite functions is the stability of the free Schr\"odinger flow, $\exp(it\Delta)$ in a truncated Hermite basis. In this work, we address precisely this point. Based on fundamental PDE estimates on the control of weighted Sobolev norms \cite{ginibre1979sigma} we justify the use of Hermite basis functions even if no trapping potential is present. {These weighted Sobolev spaces, $\Sigma^k$, are the physically natural spaces to look for solutions in the context of Hermite expansions, seeing as they are the natural analogue of standard Sobolev spaces for dispersive equations when both regularity and spatial decay must be controlled.} This allows us to use Hermite basis functions for a large class of nonlinear Schr\"odinger equations on unbounded domains.  \revisions{We emphasise that both the underlying stability estimates and our convergence analysis are formulated on $\mathbb{R}^d$ for arbitrary dimension $d\geq 1$, with the corresponding methods extending to the multivariate setting via the standard tensor-product construction. This is reflected in our numerical experiments, which cover examples in one, two and three spatial dimensions.}

In addition to controlling unboundedness in the spatial domain another numerical obstacle arises when stiffness enters the potential $V$. The latter occurs, for instance, when  the nonlinearity $V$ involves spatial derivatives, such as in derivative nonlinear Schr\"odinger equations (DNLSE) on the real line
\begin{align}\label{eqn:DNLSE2}
    i\partial_t\psi+\partial_x^2\psi-2i\delta\partial_x(|\psi|^2\psi)=0, \quad x\in\mathbb{R},
\end{align}
where $\delta\in\mathbb{R}\setminus\{0\}$, i.e. \eqref{eqn:general_schroedinger} with $V(\psi,x)=i\delta\partial_x(|\psi|^2\psi)$. Equation \eqref{eqn:DNLSE2} was first derived in \cite{mio1976modified,mjolhus1976modulational} to model the propagation of circular polarised nonlinear Alfv\'en waves in cold plasma and is also at the heart of very recent cutting-edge theoretical studies \cite{MR4448993,MR4628747}. In the second part of this paper, with the aid of a gauge transform, we introduce a fully explicit unconditionally stable Hermite spectral method for \eqref{eqn:DNLSE2}. The choice of gauge transform was originally introduced in PDE analysis by \cite{hayashi1993initial} based on work by \cite{kundu1984landau}, and has been successfully used in the theoretical analysis of \eqref{eqn:DNLSE2}  in low regularity spaces \cite{hayashi1992derivative,hayashi1994modified,MR1406687,MR1693278, MR1871414,MR1950826,herr2006cauchy}, but not yet in numerical analysis literature.

\subsection{Contents}

The rest of this manuscript is structured as follows. In \S\ref{sec:spatial_discretisation} we describe our choice of spatial discretisation, the Hermite basis, which forms the central ingredient in our construction of methods for \eqref{eqn:general_schroedinger}. In particular, in \S\ref{sec:stability_of_free_schroedinger_flow_sigma} we prove the central result underpinning this work, namely the stability of the free Schr\"odinger flow in $\Sigma^k$-norms which are the natural norms to study convergence properties of Hermite spectral methods. In \S\ref{sec:standard_cubic_NLSE} we apply this framework to the cubic nonlinear Schr\"odinger equation (NLSE) on the real line including a fully discrete error analysis of the corresponding splitting method. This is followed, in \S\ref{sec:DNLSE} by our construction of fully explicit unconditionally stable integrators for the DNLSE using the R-transform and, finally, in \S\ref{sec:numerical_examples} by numerical examples \revisions{in one, two and three spatial dimensions, evaluating our methods against Fourier spectral methods on truncated domains and against a state-of-the-art integrator for the DNLSE}.

\section{Spatial discretisation: the Hermite basis}\label{sec:spatial_discretisation}
As mentioned above, a very natural choice for the spatial discretisation of \eqref{eqn:general_schroedinger} on an unbounded domain is a Hermite spectral method (cf. \cite{ThCaNe2009}) based on the following $L^2(\R^d)$-orthonormal basis:
\begin{align*}
\mathcal{H}_m(x)=\prod_{j=1}^d\left(\mathrm{H}_{m_j}(x_j) \mathrm{e}^{-\frac{1}{2}x_j^2}\right), \quad m\in\N^{d},
\end{align*}
where $\mathrm{H}_{m_j}$ denotes the Hermite polynomial of degree $m_j\in\N$ normalised with respect to the weight $w(x)=\exp(-x^2)$, that is $H_{m_j}(y)=C_{m_j}(-1)^{m_j}e^{y^2}\partial_y^{m_j}e^{-y^2}$ with an appropriate normalisation constant $C_{m_j}>0$. The $\mathcal{H}_m$ are eigenfunctions of the Schr\"odinger operator with harmonic potential, $H=-\Delta+|x|^2$, with
\begin{align}\label{eq:eigenfunction_property}
	H(\mathcal{H}_m)=\left(d+2\sum_{j=1}^d m_j\right) \mathcal{H}_m.
\end{align}
The above relation shows why Hermite basis functions are typically used only in the case of harmonic traps, i.e. for potentials involving $\vert x\vert^2$. The natural spaces to study the convergence behaviour of Hermite spectral methods are so-called $\Sigma^k$-norms given in \cite{ginibre1979sigma} which are defined, for a given $k\in\mathbb{N}$, as
\begin{align}\label{eqn:expression_Sigma_k_regularity+decay}
    \|f\|_{\Sigma^k}:=\|f\|_{H^k(\R^d)}+ \||x|^k f\|_{L^2(\R^d)}
\end{align}
for $f\in L^2(\mathbb{R}^d)$. We will work on $\Sigma^k = \left\{ f\in L^2(\R^d),\quad
    \|f\|_{\Sigma^k}<\infty\right\}$,
and, by the equivalence of norms (see e.g. \cite{BCM08,Helffer1984})
\begin{equation}\label{eqn:details_sigma_norm_in_hermite_basis}
\|f\|_{\Sigma^k}^2\sim \sum_{m\in \N^d} \lambda_m^k |\alpha_m|^2,\quad
\lambda_m = \frac{d}{2}+\sum_{j=1}^d m_j,\quad
\text{for } f(x) = \sum_{m\in \N^d} \alpha_m  \mathcal{H}_m(x),
\end{equation}
we can equivalently regard $\Sigma^k$ as spaces of $L^2$-functions with appropriately decaying Hermite coefficients.
\begin{remark}
  In numerical analysis literature (cf. \cite{Gauckler2011}) the space $\Sigma^k$ is sometimes denoted by $\widetilde{H}^k$, and in theoretical PDE literature by $H^{k,k}$.
\end{remark}
Naturally, $\Sigma^k$ are Hilbert spaces and we list some of the central properties of $\Sigma^k$ (algebra properties, stable transforms and differentiation) in Appendices~\ref{app:properties_Sigma_k} \& \ref{sec:computing_with_hermite_expansions}. In particular, we note that the Laplacian operator can be applied stably and efficiently to a Hermite expanded function $f(x)=\sum_{m=0}^{M-1}\alpha_m\mathcal{H}_m(x)$ by computing, in one dimension,\vspace{-0.5cm}
\begin{align*}
    \bm{\delta}=-\mathsf{D}_{\lambda}\bm{\alpha}+\mathsf{T}\mathsf{D}_{|x|^2}\mathsf{T}^{-1}\bm{\alpha}
\end{align*}
where $\bm{\delta}=(\delta_0,\dots,\delta_{M-1})$ are the Hermite coefficients of $\Delta f$, $\mathsf{T}$ is the transformation matrix from function values to coefficients and $D_{\lambda},D_{|x|^2}$ are two diagonal matrices (cf. \eqref{eqn:Hermite_coeffs_of_laplacian}) and with the standard tensor-product extension to higher dimensions.

\begin{remark} An alternative way to extend these approximation spaces to the multivariate setting is using Hagedorn functions \cite[\S4]{Lasser_Lubich_2020} which provide a transported version of Hermite functions and offer an efficient way of tracking and resolving solution variations. This comes at the cost of more complex theory and, in order to emphasise the main point of this work, which is that splitting methods are feasible in the Hermite basis even without harmonic traps, we focus on the Hermite basis case instead.
\end{remark}

\subsection{Stability of the free Schr\"odinger flow in $\Sigma^k$ norms}\label{sec:stability_of_free_schroedinger_flow_sigma}
The central stability result that guarantees the safe application of Hermite spectral methods even when no harmonic potential is present is the following observation from Lemma 1.2 in \cite{ginibre1979sigma}.
\begin{prop}\label{prop:stability_free_schroedinger}For any $k\in\mathbb{N}$ there is a constant $C>0$ such that for any $u_0\in \Sigma^{k}$ and any $t\geq0$ we have
	\begin{align*}
		\left\|e^{it\Delta}u_0\right\|_{\Sigma^{k}}\leq (1+Ct)^k \|u_0\|_{\Sigma^k},
	\end{align*}
    i.e. the free Schr\"odinger flow is stable in $\|\cdot\|_{\Sigma^k}$.
\end{prop}\begin{proof} {For the sake of completeness we include the short proof of this estimate.} We begin by proving the result for $k=1$.
Let us introduce the {{so-called Galilean}} operator
$J(t) := x + 2it\nabla$, which commutes with the free Schr\"odinger differential equation:
\[
(x + 2it\nabla)(i\partial_t f + \Delta f)
=
(i\partial_t + \Delta)(xf + 2it\nabla f).
\]
Thus considering \(u(t):=e^{it\Delta}u_0\), the solution of 
$i\partial_t u + \Delta u = 0,$ with $u|_{t=0} = u_0$,
we observe that
$v(t) := J(t)u(t)$ solves
\[
\begin{cases}
i\partial_t v + \Delta v = 0, \\
v|_{t=0} = J(0)u_0 = xu_0.
\end{cases}
\]
The linear Schr\"odinger equation conserves the $L^2$-norm, i.e. $\|v(t)\|_{L^2} = \|v(0)\|_{L^2}$, hence $\|xu(t) + 2it\nabla u(t)\|_{L^2} = \|xu_0\|_{L^2}$. Therefore we have
\[
\|xu(t)\|_{L^2}
\le
\|xu_0\|_{L^2} + 2t\|\nabla u(t)\|_{L^2}.
\]
Using conservation of the \(\dot H^1\)-norm for the linear Schrödinger equation, $\|\nabla u(t)\|_{L^2} = \|\nabla u_0\|_{L^2},$ we obtain
\[
\|xu(t)\|_{L^2}
\le
\|xu_0\|_{L^2} + 2t\|\nabla u_0\|_{L^2}.
\]
Recall $\|u(t)\|_{\Sigma^1}
=
\|xu(t)\|_{L^2} + \|u(t)\|_{H^1},$ thus
\begin{align}\label{eqn:aux_identity_1}
\|u(t)\|_{\Sigma^1}
\le
\|xu_0\|_{L^2} + (2t+1)\|\nabla u_0\|_{L^2}+\|u_0\|_{H^1}
\le
(2t+1)\|u_0\|_{\Sigma^1},
\qquad t\geq 0.
\end{align}
This completes the proof for $k=1$. We prove the result for $k\geq 2$ by induction. Suppose we have already shown\vspace{-0.3cm}
\begin{align}\label{eqn:aux_identity_k-1}
    \|u(t)\|_{\Sigma^k}\leq (1+c_{k}t)^{k}\|u_0\|_{\Sigma^{k}}
\end{align}
for some $c_{k}>0$, for all data $u_0\in\Sigma^k$. This is true for $k=1$ as shown in \eqref{eqn:aux_identity_1}. Estimate \eqref{eqn:aux_identity_k-1} holds for $v(t)=J(t)u(t)$ and for $\nabla u(t)$ since both  are linear Schr\"odinger evolutions - indeed $\nabla$ and $J(t)$ commute with the free Schr\"odinger operator. Therefore we have
\begin{align}\label{eqn:induction_step1}
    \||x|^{k}(x+2it\nabla) u(t)\|_{L^2}&\leq (1+c_{k}t)^{k}\|xu_0\|_{\Sigma^{k}},\\\label{eqn:induction_step2}
    \||x|^{k}\nabla u(t)\|_{L^2}&\leq (1+c_{k}t)^{k}\|\nabla u_0\|_{\Sigma^{k}}.
\end{align}
We deduce
\begin{align*}
    \||x|^{k+1}u(t)\|_{L^2}\leq (1+c_{k}t)^{k}\|xu_0\|_{\Sigma^{k}}+2t(1+c_{k}t)^{k}\|\nabla u_0\|_{\Sigma^{k}}\leq (1+2t)(1+c_kt)^k\|u_0\|_{\Sigma^{k+1}},
\end{align*}
thus,  adding $\|u(t)\|_{H^{k+1}}$ to both sides,
\begin{align*}
   \|u(t)\|_{\Sigma^{k+1}}\leq (1+(1+2t)(1+c_kt)^k)\|u_0\|_{\Sigma^{k+1}}\leq (1+c_{k+1}t)^{k+1}\|u_0\|_{\Sigma^{k+1}},
\end{align*}
for an appropriately chosen $c_{k+1}>0$, completing the induction step. This  completes the proof.
\end{proof}
\subsection{Interpolation operator and Hermite series truncation}\label{sec:spatial_truncation_error}
This section summarises some of the basic definitions and estimates to set the stage for our analysis of the Hermite spectral method. These results are standard and can be found for example in the important prior works \cite{Gauckler2011,Thalhammer2012}. In what follows we will apply a Hermite spectral semidiscretisation in space using the Hermite quadrature points $x_0,\dots, x_{M-1}$ (with the obvious tensor product extension to $d\geq 2$). We represent functions in the form
\begin{align*}
    \psi_{M}(t,x)=\sum_{\substack{m\in\mathbb{N}^d\\0\leq m\leq M-1}}\alpha_m(t)\mathcal{H}_m(x).
\end{align*}
The space of all such functions is the Hermite approximation space $\mathcal{S}_M\subset L^2$. To deal with the nonlinearities we introduce the interpolation operator $\mathcal{Q}_{M}$ as follows (cf. Definition 3.2 in \cite{Gauckler2011}):
\begin{definition} For $u\in\Sigma^k, k\in\mathbb{N}$ we define the interpolation $\mathcal{Q}_M(u)$ to be the unique function in $\mathcal{S}_M$, i.e.
\begin{align*}
     \mathcal{Q}_M(u)(x)=\sum_{\substack{m\in\mathbb{N}^d\\0\leq m\leq M-1}}\hat{u}_m\mathcal{H}_m(x),
\end{align*}
with coefficients $\hat{u}_m$ such that $\mathcal{Q}_M(u)(x_m)=u(x_m)$ for all $m\in\mathbb{N}^d, 0\leq m\leq M-1$. The uniqueness of this function is guaranteed by exactness of Gauss--Hermite quadrature.
\end{definition}
The interpolation operator has the following properties which are central to our fully-discrete error analysis in later parts of this work.

\begin{lemma}\label{lem:Hermite_interpolation}
For $u \in \Sigma^k$ with an integer $k \ge d$ we have
\begin{equation}
\|u - \mathcal{Q}_M(u)\|_{\Sigma^{k'}}
\le
C\, M^{\frac d3 - \frac12 (k - k')}\, \|u\|_{\Sigma^k}
\end{equation}
for $k' \le k$ and with a constant $C$ depending only on $d$, $k$ and $k'$.
\end{lemma}
\begin{proof}
    See Proposition 5.1 in \cite{Gauckler2011}.
\end{proof}

\begin{remark}
    Further results concerning the approximation properties of Hermite functions are available in \cite{BOYD1984382,Hille1939HermitianSeries,Hille1940HermitianSeriesII,Lasser_Lubich_2020}.
\end{remark}

\begin{lemma}\label{lem:product_interpolation_product_estimate}
Let $u,v \in \Sigma^k$ for $k > \frac d2$. Then we have
\begin{equation}
\|\mathcal{Q}_M(uv)\|_{\Sigma^0}
\le
\sup_{\substack{m\in\mathbb{N}^d\\0\leq m\leq M-1}}|u(x_m)| \, \|\mathcal Q(v)\|_{\Sigma^0} .
\end{equation}
\end{lemma}
\begin{proof}
    See Lemma~5.2 in \cite{Gauckler2011}.
\end{proof}

\section{Hermite spectral methods for the cubic NLSE}\label{sec:standard_cubic_NLSE}
Proposition~\ref{prop:stability_free_schroedinger} immediately allows us to study the convergence properties of splitting methods for the cubic nonlinear Schr\"odinger equation (NLSE)
\begin{align}\label{eqn:cubic_NLSE}
\begin{cases}	i\partial_t\psi = -\Delta\psi +\mu|\psi|^2\psi,\\
\psi|_{t=0}=\psi_0,
\end{cases}
\end{align}
with $\mu\in\mathbb{R}\setminus\{0\}$ in $\Sigma^k$. This corresponds to \eqref{eqn:general_schroedinger} with potential $V(\psi,x)=\mu|\psi|^2$. In the following, we will study the convergence properties of the Lie splitting method, but note that this naturally extends to higher order splitting methods (similarly to the extension for Fourier spectral methods presented in \cite{koch2013error}). The method under consideration here is thus
\begin{align}\label{eqn:Lie_splitting}
	\psi^{n+1}=e^{i\tau\Delta}e^{-i\tau\mu|\psi^n|^2}\psi^n,
\end{align}
where $\tau>0$ is the time step and $\psi^n(x)$ approximates the exact solution $\psi(t,x)$ at time $t= n \tau$, i.e., $\psi^n(x)\approx \psi(n\tau,x)$.
\subsection{Semi-discrete convergence analysis}\label{sec:semi-discrete_temporal_analysis_cubic_NLSE}
Let us first focus on the temporal discretisation \eqref{eqn:Lie_splitting}. The central convergence result in this case is the following.

\begin{theorem}\label{thm:semi-discrete_error_analysis_cubicNLSE}
	Let $k> d/2$ be an integer. Suppose $\psi(t,x)$ is the exact solution of \eqref{eqn:cubic_NLSE}. Then there exists a $\tau_0>0$ such that for all $0 < \tau < \tau_0$,
    \begin{align*}
        \|\psi^n-\psi(t_n,\cdot)\|_{\Sigma^k}\leq C \tau,\quad \text{for all\ }0\leq t_n=n\tau\leq T
    \end{align*}
    with a constant $C>0$ depending only on $T,\sup_{t\in [0,T]}\|\psi(t,\cdot)\|_{\Sigma^{k+2}}, d, k$. 
\end{theorem}
Note as an immediate corollary we have a control on the $L^2$ and $H^1$ convergence properties of the semi-discrete Lie splitting.
\begin{cor}\label{cor:L2estimate_semi-discrete}Let $k_0>d/2+2, k_1>\max\{d/2+2,3\}$ be integers and $M_k:=\sup_{t\in [0,T]}\|\psi(t,\cdot)\|_{\Sigma^{k}}$. Then
\begin{align*}
        \|\psi^n-\psi(t_n,\cdot)\|_{L^2}\leq C_{k_0} \tau,\quad \text{for all\ }0\leq t_n=n\tau\leq T,\\
        \|\psi^n-\psi(t_n,\cdot)\|_{H^1}\leq C_{k_1} \tau,\quad \text{for all\ }0\leq t_n=n\tau\leq T,
\end{align*}
where $C_k>0$ ($k\in\{k_0,k_1\}$) are constants that depend only on $k,T,\sup_{t\in[0,T]}\|\psi(t,\cdot)\|_{\Sigma^{k}}$.
\end{cor}
\begin{proof}
    This is an immediate consequence of Theorem~\ref{thm:semi-discrete_error_analysis_cubicNLSE}, noting that $\|u\|_{\Sigma^k}=\|x^ku\|_{L^2}+\|u\|_{H^k}$.
\end{proof}
The proof of Theorem~\ref{thm:semi-discrete_error_analysis_cubicNLSE} relies on two standard steps (consistency and stability) as per the following two lemmas.
\begin{lemma}[Stability]\label{lem:stability_cubic_nlse} Let $k>d/2$ be an integer and $w,v\in \Sigma^k$. Then there is a constant $C>0$ depending only on $\|w\|_{\Sigma^k},\|v\|_{\Sigma^k}, d, k$ such that, for all $\tau>0$,
\begin{align*}
    \left\|e^{i\tau\Delta}e^{-i\tau\mu|w^n|^2}w^n-e^{i\tau\Delta}e^{-i\tau\mu|v^n|^2}v^n\right\|_{\Sigma^k}\leq e^{C\tau}\|w-v\|_{\Sigma^k}.
\end{align*}
\end{lemma}
\begin{proof}
    Using Proposition~\ref{prop:stability_free_schroedinger}, there is a constant $C_1>0$ depending only on $k,d$ such that 
    \begin{align*}
        \left\|e^{i\tau\Delta}e^{-i\tau\mu|w^n|^2}w^n-e^{i\tau\Delta}e^{-i\tau\mu|v^n|^2}v^n\right\|_{\Sigma^k}\leq e^{C_1\tau}\left\|e^{-i\tau\mu|w^n|^2}w^n-e^{-i\tau\mu|v^n|^2}v^n\right\|_{\Sigma^k}.
    \end{align*}
    And the desired stability estimate follows then analogously to the proof of Lemma 2.2 in \cite{Gauckler2011} using Gronwall's lemma on 
    \begin{align*}
        i\frac{\partial}{\partial t}\theta&=\mu|w|^2\theta,\quad\theta(0)=w,\quad\text{and}\quad
        i\frac{\partial}{\partial t}\eta=\mu|v|^2\eta,\quad\eta(0)=v.
    \end{align*}
\end{proof}
\begin{lemma}[Local error - consistency]\label{lem:consistency_cubic_nlse} Let $k>d/2$ be an integer and denote by $\psi(t,x)$ the exact solution of \eqref{eqn:cubic_NLSE}. Then there is a constant $C>0$ depending only on $\sup_{t\in[0,\tau]}\|\psi(t,\cdot)\|_{\Sigma^{k+2}}$ such that
\begin{align*}
    \|e^{i\tau\Delta}e^{-i\tau\mu|\psi^n|^2}\psi^n-\psi(\tau)\|_{\Sigma^{k}}\leq C \tau^2.
\end{align*}
\end{lemma}
\begin{proof}
    The result can be proven using analogous arguments as in the Fourier spectral case \cite[Section 4.4]{lubich2008splitting}, by expressing the error in terms of the Lie commutator
    \begin{align*}
        [\hat{T},\hat{V}](\psi)=\hat{T}'(\psi)\hat{V}(\psi)-\hat{V}'(\psi)\hat{T}(\psi),
    \end{align*}
    where $\hat{T}(\psi)=i\Delta\psi$, $\hat{V}(\psi)=-i\mu|\psi|^2\psi$.
\end{proof}

\begin{proof}[Proof of Theorem~\ref{thm:semi-discrete_error_analysis_cubicNLSE}]
The global convergence result follows thus by a simple Lady-Windermere's fan argument combining Lemma~\ref{lem:stability_cubic_nlse} and Lemma~\ref{lem:consistency_cubic_nlse}.
\end{proof}

\subsection{Fully-discrete convergence analysis}
The fully discrete scheme is\vspace{-0.4cm}
\begin{align}\begin{split}\label{eqn:Lie_splitting_fully_discrete}
	\psi^{n+1}_M&=e^{i\tau\Delta}\mathcal{Q}_M\!\left(e^{-i\tau\mu|\psi_M^n|^2}\psi_M^n\right),\\
    \psi^{0}_M&=\mathcal{Q}_M(\psi_0).\end{split}
\end{align}
Based on the spatial discretisation estimates in \S\ref{sec:spatial_truncation_error} and the semi-discrete analysis in \S\ref{sec:semi-discrete_temporal_analysis_cubic_NLSE} we can establish the following global error estimate.
\begin{theorem}\label{thm:fully_discrete_error_analysis_cubicNLSE}
Let $k>\lceil\frac{d+1}{2}\rceil+2+\frac{2d}{3}$ be an integer. Denote by
$M_{k+2}:=\sup_{t\in[0,T]}\|\psi(t)\|_{\Sigma^{k+2}}$. Let $\psi_M^n$ be generated by the fully discrete scheme \eqref{eqn:Lie_splitting_fully_discrete} with step size $\tau>0$,
then there exist $\tau_0>0,M_0>0$ and $C>0$  depending only on $d,k,s,T$ and $M_{k+2}$ such that for all $0 <\tau < \tau_0, M_0\leq M$ and $0\le t_n\le T$
\begin{equation}\label{eq:full_error_L2}
\|\psi_M^n-\psi(t_n)\|_{L^2}
\;\le\;
C\Bigl(\tau + M^{1+\frac d3-\frac{k}{2}}\Bigr), \quad \text{for\ }0\leq t_n=nh\leq T.
\end{equation}
\end{theorem}
Let us define the one-step maps
\[
\Phi_\tau(u):=e^{i\tau\Delta}\bigl(e^{-i\tau\mu|u|^2}u\bigr),
\qquad
\Phi_{\tau,M}(u):=e^{i\tau\Delta}\,\mathcal Q_M\!\bigl(e^{-i\tau\mu|u|^2}u\bigr).
\]
For the proof we will use the following two lemmas.
\begin{lemma}[Stability of fully discrete scheme]\label{lem:stability_fully_discrete_NLSE} Let $k>d/2$ be an integer, then there is a constant $C>0$ such that for all $u,v\in\mathcal{S}_M$ we have
\begin{align*}
    \|\Phi_{\tau,M}(u)-\Phi_{\tau,M}(v)\|_{L^2}\leq e^{C(\|u\|_{\Sigma^k}^2+\|v\|_{\Sigma^k}^2)\tau}\|u-v\|_{L^2}.
\end{align*}
\end{lemma}
\begin{proof}
    By Proposition~\ref{prop:stability_free_schroedinger},
\[
\|\Phi_{\tau,M}(u)-\Phi_{\tau,M}(v)\|_{\Sigma^k}
\leq 
e^{\tau C_0}\left\|\mathcal Q_M\!\Bigl(e^{-i\tau\mu|u|^2}u-e^{-i\tau\mu|v|^2}v\Bigr)\right\|_{\Sigma^k}.
\]
Pointwise, the map $F(w):=e^{-i\tau\mu|w|^2}w$ is Lipschitz: for $\tau_0$ sufficiently small we have for any $0<\tau<\tau_0$
\begin{align}\label{eqn:pointwise_lipschitz_bound}
|F(u)-F(v)|
\le \bigl(1+C\tau(|u|^2+|v|^2)\bigr)\,|u-v|
\end{align}
for a constant $C$ depending only on $|\mu|$. Hence
\[
F(u)-F(v)=a_{u,v}\,(u-v),
\qquad
|a_{u,v}(x)|\le 1+C\tau\bigl(|u(x)|^2+|v(x)|^2\bigr).
\]
Using Lemma~\ref{lem:product_interpolation_product_estimate} on
$a_{u,v}(u-v)$, and noting that $\mathcal Q_M(u-v)=u-v$ on the Hermite space $\mathcal{S}_M$, we obtain
\[
\|\Phi_{\tau,M}(u)-\Phi_{\tau,M}(v)\|_{L^2}
\le
\sup_{0\le m\le M-1}|a_{u,v}(x_m)|\,\|u-v\|_{L^2}.
\]
Since $k>\frac d2$ implies $\Sigma^k\hookrightarrow L^\infty$ (standard Sobolev embedding noting that $\Sigma^k\hookrightarrow H^k$),
we have $\sup_m|u(x_m)|\le \|u\|_{L^\infty}\lesssim \|u\|_{\Sigma^k}$ and similarly for $v$.
\end{proof}
\begin{lemma}[Local error of the fully discrete scheme]\label{lem:local_error_fully_discrete_NLSE}
Let $k>d/2$ be an integer, and let $k'=\lceil\frac{d+1}{2}\rceil$ be the smallest integer greater than $d/2$. If $\phi\in \Sigma^k$ then
\begin{align*}
    \|\Phi_{\tau,M}(\mathcal{Q}_{M}(\phi))-\mathcal{Q}_{M}(\Phi_{\tau}(\phi))\|_{L^2}\leq C\tau M^{1+\frac{d}{3}-\frac{k}{2}}\left(e^{C\|\mathcal{Q}_{M}(\phi)\|_{\Sigma^{k'}}\|\phi\|_{\Sigma^{k'}}\tau}\|\phi\|_{\Sigma^{k}}+\|\Phi_{\tau}(\phi)\|_{\Sigma^k}\right),
\end{align*}
where $C$ only depends on $d,k$.
\end{lemma}
\begin{proof}
    See Proposition 6.2 in \cite{Gauckler2011}.
\end{proof}
\begin{proof}[Proof of Theorem~\ref{thm:fully_discrete_error_analysis_cubicNLSE}] We have
\begin{align*}
    \|\psi_M^{n}-\psi(t_n)\|_{L^2}\leq \underbrace{\|\psi_M^n-\mathcal{Q}_{M}(\psi^n)\|_{L^2}}_{\text{(I)}}+\underbrace{\|\mathcal{Q}_{M}(\psi^n)-\psi^n\|_{L^2}}_{\text{(II)}}+\underbrace{\|\psi^n-\psi(t_n)\|_{L^2}}_{\text{(III)}}.
\end{align*}
For (III) we have the following bound from Corollary~\ref{cor:L2estimate_semi-discrete}
\begin{align*}
    \text{(III)}=\|\psi^n-\psi(t_n)\|_{L^2}\leq C_{0}\tau,
\end{align*}
where $C_{0}$ depends on $d,k$ and $\sup_{t\in[0,T]}\|\psi(t)\|_{\Sigma^k}$.
For (II) we can use Lemma~\ref{lem:Hermite_interpolation} together with the boundedness of the semi-discrete numerical solution in $\Sigma^k$ implied by Theorem~\ref{thm:semi-discrete_error_analysis_cubicNLSE} to show that there are constants $C_1,\tau_0,M_0$ depending only on $M_{k+2},d,k,s,T$ such that
\begin{align*}
    \text{(II)}=\|\mathcal{Q}_{M}(\psi^n)-\psi^n\|_{L^2}\leq C M^{\frac{d}{3}-\frac k2}\|\psi^n\|_{\Sigma^k}\leq C_1M^{\frac{d}{3}-\frac k2}
\end{align*}
Thus it remains to control (I). For this we follow a similar Lady Windermere's fan argument as presented in the proof of Theorem 3.4 in \cite{Gauckler2011}. We note that
\begin{align*}
    \text{(I)}&=\|\mathcal{Q}_{M}(\Phi^n_{\tau}\psi_0)-\Phi_{\tau,M}^n(\mathcal{Q}_{M}(\psi_0))\|_{L^2}\\
    &\leq \left\|\mathcal{Q}_{M}\!\left(\Phi_{\tau}(\Phi^{n-1}_{\tau}\psi_0)\right)-\Phi_{\tau,M}\!\left(\mathcal{Q}_{M}(\Phi^{n-1}_{\tau}\psi_0)\right)\right\|_{L^2}+\left\|\Phi_{\tau,M}\!\left(\mathcal{Q}_{M}(\Phi^{n-1}_{\tau}\psi_0)\right)-\Phi_{\tau,M}\!\left(\Phi_{\tau,M}^{n-1}\mathcal{Q}_{M}(\psi_0)\right)\right\|_{L^2}
\end{align*}
Taking $k'=\lceil\frac{d+1}{2}\rceil$ and applying Lemma~\ref{lem:stability_fully_discrete_NLSE} we find, writing 
\begin{align*}
L_n
&:= \exp\!\Big(
C\big(
\|\mathcal{Q}_M(\Phi_\tau^{n}\psi_0)\|_{\Sigma^{k'}}^2
+
\|\Phi_{\tau,M}^{n}(\mathcal{Q}_M\psi_0)\|_{\Sigma^{k'}}^2
\big)\tau
\Big),\\
\delta_n &:= \bigl\|\mathcal{Q}_M(\Phi_\tau^{n+1}\psi_0)-\Phi_{\tau,M}\bigl(\mathcal{Q}_M(\Phi_\tau^{n}\psi_0)\bigr)\bigr\|_{L^2},
\end{align*}
that
\[
\text{(I)} \le \sum_{j=0}^{n-1}
\left(\prod_{m=j+1}^{n-1} L_m\right)\,\delta_j.
\]
Combining this with Lemma~\ref{lem:local_error_fully_discrete_NLSE} we obtain, after a few steps as in \cite[p. 413]{Gauckler2011},
\begin{align*}
    \text{(I)}\leq C M^{1+\frac{d}{3}-\frac{k}{2}}\frac{e^{C\tilde{a}(n)^2n\tau}-1}{C\tilde{a}(n)^2},
\end{align*}
where $C$ depends on $d,k,t_n$ and $\sup_{t\in[0,t_n]}\|\psi(t)\|_{k+2}$, and $\tilde{a}(n)=\max_{0\leq j\leq n-1, 0\leq i\leq n-j-1}\|\Phi_{\tau,M}^i\!\left(\mathcal{Q}_{M}(\psi^j)\right)\|_{\Sigma^{k''}}$. In the above we used the additional estimate $\|\mathcal{Q}_{M}(f)\|_{\Sigma^{k'}}\leq C\|f\|_{\Sigma^{k'+2d/3}}\leq C\|f\|_{\Sigma^{k}}$ for $f\in \Sigma^{k}$ which follows from Lemma~\ref{lem:Hermite_interpolation}. It remains to control $\tilde{a}_n$. For this we note that for $f\in\mathcal{S}_M$ we have 
\begin{align*}
    \|f\|_{\Sigma^{k'}}\leq C M^{\frac{k'}{2}}\|f\|_{L^2},
\end{align*}
for a constant $C>0$ independent of $f,M$. Therefore we have
\begin{align*}
\|\Phi_{\tau,M}^n(\mathcal{Q}_{M}(\psi_0))-\mathcal{Q}_{M}(\Phi^n_{\tau}\psi_0)\|_{\Sigma^{k'}}\leq C M^{1+\frac{d}{3}+\frac{k'}{2}-\frac{k}{2}}\frac{e^{C\tilde{a}(n)^2n\tau}-1}{C\tilde{a}(n)^2}.
\end{align*}
Taking $M$ sufficiently large we can thus control $\tilde{a}(n)\leq 2C$ uniformly in $n$ and conclude that
\begin{align*}
    \text{(I)}\leq C M^{1+\frac{d}{3}-\frac{k}{2}},
\end{align*}
with a constant $C$ that depends on $\sup_{t\in[0,t_n]}\|\psi(t)\|_{\Sigma^k}, d, k$. Combining the estimates for (I),(II),(III) completes the proof.
\end{proof}

\section{Stable simulation of the DNLSE}\label{sec:DNLSE}
We aim to solve the derivative nonlinear Schr\"odinger equation (DNLSE)
	\begin{align}\label{eqn:DNLSE}
    i\partial_t\psi+\partial_x^2\psi-2i\delta\partial_x(|\psi|^2\psi)=0, \quad x\in\mathbb{R},
	\end{align}
i.e. \eqref{eqn:general_schroedinger} with $V(\psi,x)=i\delta\partial_x(|\psi|^2\psi)$. The stiff nonlinearity generally causes stability issues in classical methods, and leads to the imposition of stringent CFL conditions. In our new approach we exploit the so-called R-transform to derive fully explicit unconditionally stable methods (with respect to time step size and spatial resolution). 

\subsection{The R-transform and stable algorithm for the DNLSE}\label{sec:stable_algorithm_DNLSE}
The basis for our new stable algorithm is the R-transform introduced in \cite{hayashi1993initial} based on work by \cite{kundu1984landau}. For this we introduce the new variables
\begin{align}\label{eqn:transformation}
	u:=E^2\psi,\quad v:=E\partial_x(E\psi),\quad E(t,x)=\exp\left(i\delta\int_{-\infty}^x|\psi(t,y)|^2dy\right),
\end{align}
which makes \eqref{eqn:DNLSE} equivalent to the following coupled system:
\begin{align}\begin{split}\label{eqn:coupled_system}
		\partial_t u&=i\partial_x^2 u+2u^2\bar{v},\quad
		\partial_t v=i\partial_x^2 v-2v^2\bar{u}.
	\end{split}
\end{align}
This corresponds to a coupled system of Schr\"odinger equations of the form \eqref{eqn:general_schroedinger} with complex potentials $V_1(u,v,x)=i2u^2\bar{v}$ and $V_2(u,v,x)=-i2v^2\bar{u}$. We can therefore exploit \eqref{eqn:transformation} to construct the following unconditionally stable method for \eqref{eqn:DNLSE}.

\paragraph{1. Transforming the initial conditions.}
Given initial data $\psi_0(x)$ for \eqref{eqn:DNLSE}, we first compute the
gauge factor
\[
E_0(x)
=
\exp\!\left(
i\delta \int_{-\infty}^x |\psi_0(y)|^2\,dy
\right),
\]
where the spatial integral is evaluated spectrally using the Hermite basis, as described in Appendix~\ref{app:stable_integration}. The initial conditions for the transformed variables are then defined by
\begin{align*}
u_0(x) &= E_0(x)^2\,\psi_0(x), \\
v_0(x) &= E_0(x)\,\partial_x\!\bigl(E_0(x)\psi_0(x)\bigr).
\end{align*}
The spatial derivative is computed spectrally using
the Hermite differentiation operator $\mathsf{A}_{\partial}$ given in \eqref{eqn:differentiation_operator}.

\paragraph{2. Time discretisation and splitting scheme.}
To numerically integrate the coupled system \eqref{eqn:coupled_system}, we employ a splitting method. Note in principle this can be done at arbitrary order, but for now we focus on the second-order Strang splitting method in time. Writing the system abstractly as
\[
\partial_t 
\begin{pmatrix}
u \\ v
\end{pmatrix}
=
\mathcal{L}
\begin{pmatrix}
u \\ v
\end{pmatrix}
+
\mathcal{N}
\begin{pmatrix}
u \\ v
\end{pmatrix},
\]
we decompose the evolution into a linear dispersive part $\mathcal{L}$ and a nonlinear coupling part $\mathcal{N}$, where
\[
\mathcal{L}
\begin{pmatrix}
u \\ v
\end{pmatrix}
=
\begin{pmatrix}
i\partial_x^2 u \\
i\partial_x^2 v
\end{pmatrix},\quad
\mathcal{N}
\begin{pmatrix}
u \\ v
\end{pmatrix}
=
\begin{pmatrix}
2u^2\bar v \\
-2v^2\bar u
\end{pmatrix}.
\]
We note that the nonlinear coupled system
\begin{align}\label{eqn:nonlinear_system_uv}
\partial_t u&=2u^2\bar{v},\ 
\partial_t v=-2v^2\bar{u},
\end{align}
has the exact solution
\begin{align*}
    \begin{pmatrix}
        u(t)\\
        v(t)
    \end{pmatrix}=\begin{pmatrix}
        e^{t2u_0\bar{v}_0}&0\\
        0&e^{-2t\bar{u}_0v_0}
    \end{pmatrix}\begin{pmatrix}
        u_0\\
        v_0
    \end{pmatrix},
\end{align*}
since $u\bar{v}$ is conserved in \eqref{eqn:nonlinear_system_uv}. Thus we can define the Strang splitting scheme as follows
\begin{align*}
    \begin{pmatrix}
    u^{n+\frac12}\\
    v^{n+\frac12}
    \end{pmatrix}&=\begin{pmatrix}
        e^{i\frac{\tau}{2}\Delta}&0\\
        0&e^{i\frac{\tau}{2}\Delta}
    \end{pmatrix}\begin{pmatrix}
        u^n\\
        v^n
    \end{pmatrix},\\
    \begin{pmatrix}
    u^{n+1}\\
    v^{n+1}
    \end{pmatrix}&=\tilde{\Phi}_\tau\begin{pmatrix}
    u^{n}\\
    v^{n}
    \end{pmatrix}:=\begin{pmatrix}
        e^{i\frac{\tau}{2}\Delta}&0\\
        0&e^{i\frac{\tau}{2}\Delta}
    \end{pmatrix}
    \begin{pmatrix}
        e^{\tau2u^{n+\frac12}\bar{v}^{n+\frac12}}&0\\
        0&e^{-2\tau\bar{u}^{n+\frac12}v^{n+\frac12}}
    \end{pmatrix}\begin{pmatrix}
    u^{n+\frac12}\\
    v^{n+\frac12}
    \end{pmatrix}.
\end{align*}
\paragraph{3. Reconstruction of $\psi$.}
After completing the time integration for $(u,v)$, the original variable $\psi$ is recovered via the inverse R-transform $\psi(t,x) = E(t,x)^{-2} u(t,x)$, where $E$ is computed numerically from
\[
E(t,x) = \exp\!\left( i\delta \int_{-\infty}^x |u(t,y)|^2\,dy \right),
\]
using the algorithm described in Appendix~\ref{app:stable_integration}, since $|u|=|\psi|$. 

\subsection{Semi-discrete convergence analysis of the DNLSE algorithm}\label{sec:semi-discrete_analysis_DNLSE}
\begin{theorem}\label{thm:semi-discrete_global_error_DNLSE}
	Let $k> 1/2+5$ be an integer. Suppose $\psi(t,x)$ is the exact solution of \eqref{eqn:DNLSE}. Then there exists a $\tau_0>0$ such that for all $0 < \tau < \tau_0$,
    \begin{align*}
        \|\psi^n-\psi(t_n,\cdot)\|_{L^2}\leq C \tau^2,\quad \text{for all\ }0\leq t_n=n\tau\leq T
    \end{align*}
    with a constant $C>0$ depending only on $T,\sup_{t\in [0,T]}\|\psi(t,\cdot)\|_{\Sigma^{k}}, d, k$. 
\end{theorem}
To prove this convergence estimate we need, similarly to \S\ref{sec:standard_cubic_NLSE}, the following auxiliary estimates.
\begin{prop}\label{prop:stability_dnlse} Let $k>1/2$ be an integer and $w_1,w_2,v_1,v_2\in \Sigma^k$. Then there is a constant $C>0$ depending only on $\|w_1\|_{\Sigma^k},\|w_2\|_{\Sigma^k},\|v_1\|_{\Sigma^k},\|v_2\|_{\Sigma^k}, d, k$ such that for all $\tau>0$
\begin{align*}
    \left\|\tilde{\Phi}_\tau\begin{pmatrix}
    w_1\\
    w_2
    \end{pmatrix}-\tilde{\Phi}_\tau\begin{pmatrix}
    v_1\\
    v_2
    \end{pmatrix}\right\|_{\Sigma^k}\leq e^{C\tau}\left\|\begin{pmatrix}
    w_1\\
    w_2
    \end{pmatrix}-\begin{pmatrix}
    v_1\\
    v_2
    \end{pmatrix}\right\|_{\Sigma^k}.
\end{align*}
\end{prop}

\begin{proof}
    We proceed similarly to the proof of Lemma~\ref{lem:stability_cubic_nlse}. Using Proposition~\ref{prop:stability_free_schroedinger} we have that there is a constant $C_1>0$ depending only on $k,d$ such that 
    \begin{align*}
        \left\|\begin{pmatrix}
        e^{-i\frac{\tau}{2}\Delta}&0\\
        0&e^{-i\frac{\tau}{2}\Delta}
    \end{pmatrix}\left[\begin{pmatrix}
    w_1\\
    w_2
    \end{pmatrix}-\begin{pmatrix}
    v_1\\
    v_2
    \end{pmatrix}\right]\right\|_{\Sigma^k}\leq e^{\tau C_1}\left\|\begin{pmatrix}
    w_1\\
    w_2
    \end{pmatrix}-\begin{pmatrix}
    v_1\\
    v_2
    \end{pmatrix}\right\|_{\Sigma^k}.
    \end{align*}
    Thus it is sufficient to prove
  \begin{align*}
     &\left\|\begin{pmatrix}
        e^{i\frac{\tau}{2}\Delta}&0\\
        0&e^{i\frac{\tau}{2}\Delta}
    \end{pmatrix}
    \begin{pmatrix}
        e^{\tau2w_1\bar{w}_2}&0\\
        0&e^{-2\tau\bar{w}_1w_2}
    \end{pmatrix}\begin{pmatrix}
    w_1\\
    w_2
    \end{pmatrix}-\begin{pmatrix}
        e^{i\frac{\tau}{2}\Delta}&0\\
        0&e^{i\frac{\tau}{2}\Delta}
    \end{pmatrix}
    \begin{pmatrix}
        e^{\tau2v_1\bar{v}_2}&0\\
        0&e^{-2\tau\bar{v}_1v_2}
    \end{pmatrix}\begin{pmatrix}
    v_1\\
    v_2
    \end{pmatrix}\right\|_{\Sigma^k}\\&\hspace{12cm}\leq e^{\tau C_2}\left\|\begin{pmatrix}
    w_1\\
    w_2
    \end{pmatrix}-\begin{pmatrix}
    v_1\\
    v_2
    \end{pmatrix}\right\|_{\Sigma^k},
    \end{align*}
    for a constant $C_2>0$ which only depends on $\|w_1\|_{\Sigma^k},\|w_2\|_{\Sigma^k},\|v_1\|_{\Sigma^k},\|v_2\|_{\Sigma^k}, d, k$. This estimate follows immediately by Proposition~\ref{prop:stability_free_schroedinger} and by using Gronwall's lemma on\vspace{-0.25cm} 
\begin{align*}
\begin{cases}\partial_t \eta_1 &= 2 w_1\overline{w_2}\eta_1, \\
\partial_t \eta_2 &= -2 w_2 \overline{w_1}\eta_2,\end{cases}\quad\text{and}\quad\begin{cases}
\partial_t \theta_1 &= 2 v_1 \overline{v_2}\theta_1, \\
\partial_t \theta_2 &= -2 v_2,\overline{v_1}\theta_2.
\end{cases}
\end{align*}
\end{proof}
\begin{prop}\label{prop:local_error_dnlse}
 Let $k>1/2$ be an integer and denote by $u(t,x),v(t,x)$ the exact solution of \eqref{eqn:coupled_system}. Then there is a constant $C>0$ depending only on $\sup_{t\in[0,\tau]}\|(u(t,\cdot),v(t,\cdot))\|_{\Sigma^{k+4}}$ such that
\begin{align*}
    \left\|\tilde{\Phi}_\tau\begin{pmatrix}
        u_0\\
        v_0
    \end{pmatrix}-\begin{pmatrix}
        u(\tau)\\
        v(\tau)
    \end{pmatrix}\right\|_{\Sigma^{k}}\leq C \tau^3.
\end{align*}
\end{prop}
\begin{proof}This proof can be performed analogously to the Fourier spectral case, cf. Section 5.2 in \cite{lubich2008splitting}.   
\end{proof}

\begin{proof}[Proof of Theorem~\ref{thm:semi-discrete_global_error_DNLSE}]\ \\
\textbf{Step 1: Relate $\|E(f)^2f\|_{\Sigma^k}$ to $\|f\|_{\Sigma^k}$.} For a given function $f\in\Sigma^s, s>1/2, s\in\mathbb{N}$, let $E(f) = \exp\!\left( i\delta \int_{-\infty}^x |f(y)|^2\,dy \right)$. Then we have $\|x^sE(f)^2f\|_{L^2}=\|x^sf\|_{L^2}$, and
\begin{align*}
    \|\partial_x (E(f)^2f)\|_{L^2}&=\| (E(f)^2)\partial_xf\|_{L^2}+ \|f\partial_x (E(f)^2)\|_{L^2}=\|\partial_xf\|_{L^2}+\|f\partial_x (E(f)^2)\|_{L^2}\\
    &=\|\partial_xf\|_{L^2}+\|E(f)^2f 2i\delta |f|^2\|_{L^2}\leq \|f\|_{H^1}+2|\delta|\|f\|_{H^1}^3.
\end{align*}
Analogously, we have
    \begin{align*}
\|\partial_x^s (E(f)^2 f)\|_{L^2}
&\le \sum_{j=0}^s \binom{s}{j}\,\|\partial_x^j(E(f)^2)\|_{L^\infty}\,\|\partial_x^{s-j}f\|_{L^2}\le C_s\Bigl(1+\sum_{m=1}^s |\delta|^m \|f\|_{H^s}^{2m}\Bigr)\,\|f\|_{H^s},
\end{align*}
for some constants $C_s>0$ which only depend on $s$. Thus, there is a constant $C>0$ independent of $f$ (but depending on $\delta, k$) such that
\begin{align}\label{eqn:swap_dependency_u_psi}
    \|E(f)^2f\|_{\Sigma^k}\leq C \left(1+\|f\|^{2k}_{\Sigma^k}\right)\|f\|_{\Sigma^k}.
\end{align}
In the following, this allows us to note that any constant that depends on $\|u\|_{\Sigma^k}$ equivalently depends on $\|\psi\|_{\Sigma^k}$. Analogously we can show that $\|v\|_{\Sigma^k}$ is bounded above by a polynomial function of $\|\psi\|_{\Sigma^{k+1}}$, noting that the increased degree stems from the derivative in the expression for $v$ in \eqref{eqn:transformation}.\\
\textbf{Step 2: Study convergence of $u^n$.} Using a standard Lady Windermere's fan argument we can combine Proposition~\ref{prop:stability_dnlse}, Proposition~\ref{prop:local_error_dnlse} and \eqref{eqn:swap_dependency_u_psi} to show that there is a constant $C,\tau_0>0$ depending only on $\sup_{t\in[0,t_n]}\|\psi(t)\|_{\Sigma^k},k,\delta$ such that for any $0<\tau<\tau_0$
\begin{align}\label{eqn:estimate_on_u}
    \|u^n-u(t_n)\|_{L^2}\leq C \tau^2.
\end{align}
\textbf{Step 3: Connect to $\psi^n$.}
We have
\begin{align}\nonumber
    \|\psi^n-\psi(t_n,\cdot)\|_{L^2}&=\left\|e^{i\delta\int_{-\infty}^{x}|u^n(y)|^2-|u(t_n,y)|^2dy}u^n-u(t_n)\right\|_{L^2}\\\label{eqn:split_of_psi}
    &\leq \underbrace{\left\|e^{i\delta\int_{-\infty}^{x}|u^n(y)|^2-|u(t_n,y)|^2dy}u^n-u^n\right\|_{L^2}}_{\mathrm{(A)}}+\underbrace{\left\|u^n-u(t_n)\right\|_{L^2}}_{\mathrm{(B)}}.
\end{align}
For (A) we note that we have the following pointwise estimate using Cauchy--Schwarz
\begin{align}\nonumber
    \left|e^{i\delta\int_{-\infty}^{x}|u^n(y)|^2-|u(t_n,y)|^2dy}-1\right|&\leq \left|\delta\int_{-\infty}^{x}|u^n(y)|^2-|u(t_n,y)|^2dy\right|^2\\\nonumber
    &=|\delta|\left|\int_{-\infty}^{x}\left(u^n(y)-u(t_n,y)\right)\left(u^n(y)+u(t_n,y)\right)dy\right|^2\\\label{eqn:central_estimate_for_twisting_back}
    &\leq |\delta|\left\|u^n(y)-u(t_n,y)\right\|_{L^2}^2\left(\|u^n\|_{L^2}^2+\|u(t_n)\|_{L^2}^2\right)\leq \tilde{C}\tau^2,
\end{align}
where $\tilde{C}$ depends only on $\sup_{t\in[0,T]}\|u(t)\|_{\Sigma^k}$. Thus we can estimate
\begin{align}\nonumber
    \mathrm{(A)}&=\left\|\left(e^{i\delta\int_{-\infty}^{x}|u^n(y)|^2-|u(t_n,y)|^2dy}-1\right)u^n\right\|_{L^2}\\\label{eqn:estimate_of_A}
    &\leq \|u^n\|_{L^2}\sup_{x\in\mathbb{R}}\left|e^{i\delta\int_{-\infty}^{x}|u^n(y)|^2-|u(t_n,y)|^2dy}-1\right|\leq \tilde{\tilde{C}}\tau^2\|u^n\|_{L^2}.
\end{align}
Combining \eqref{eqn:split_of_psi}, \eqref{eqn:estimate_on_u} and \eqref{eqn:estimate_of_A} then yields the desired result.
\end{proof}
\subsection{Fully-discrete convergence analysis of the DNLSE algorithm}\label{sec:fully-discrete_analysis_DNLSE}
The fully discrete scheme is initialised with
\georg{\begin{align}\label{eqn:true_fully_discrete_initial_cond1}
E_{0,M}(x)&=\exp\!\left( i\delta \int_{-\infty}^x \mathcal{Q}_M(|\psi_0(t,y)|^2)\,dy \right)\\\label{eqn:true_fully_discrete_initial_cond2}
u^0_{M} &= \mathcal{Q}_{M}\!\left(\left(E_{0,M}\right)^2\psi_0\right), \\\label{eqn:true_fully_discrete_initial_cond3}
v^0_{M} &= \mathcal{Q}_{M}\!\left(i\delta \left(E_{0,M}\right)^2\psi_0|\psi_0|^2+\left(E_{0,M}\right)^2\,\partial_x\!\bigl(\psi_0\bigr)\right).
\end{align}
where in the expression for $v^0_{M}$ we used the chainrule to reduce the number of interpolation steps required. We then proceed with the temporal evolution of the form}
    \begin{align}\begin{split}\label{eqn:fully_discrete_uv_system}
    \begin{pmatrix}
    u_M^{n+\frac12}\\
    v_M^{n+\frac12}
    \end{pmatrix}&=\begin{pmatrix}
        e^{i\frac{\tau}{2}\Delta}&0\\
        0&e^{i\frac{\tau}{2}\Delta}
    \end{pmatrix}\begin{pmatrix}
        u_M^n\\
        v_M^n
    \end{pmatrix},\\
    \begin{pmatrix}
    u_M^{n+1}\\
    v_M^{n+1}
    \end{pmatrix}&=\tilde{\Phi}_{\tau,M}\begin{pmatrix}
    u_M^{n}\\
    v_M^{n}
    \end{pmatrix}:=\begin{pmatrix}
        e^{i\frac{\tau}{2}\Delta}&0\\
        0&e^{i\frac{\tau}{2}\Delta}
    \end{pmatrix}\mathcal{Q}_M\!\left[
    \begin{pmatrix}
        e^{\tau2u_m^{n+\frac12}\bar{v}_M^{n+\frac12}}&0\\
        0&e^{-2\tau\bar{u}_M^{n+\frac12}v_M^{n+\frac12}}
    \end{pmatrix}\begin{pmatrix}
    u_M^{n+\frac12}\\
    v_M^{n+\frac12}
    \end{pmatrix}\right].
    \end{split}
\end{align}
And finally
\begin{align}\label{eqn:fully_discrete_map_u_to_psi}
    \psi_M^n=\exp\!\left(-2
i\delta \int_{-\infty}^x \georg{\mathcal{Q}_M}(|u^n_M(y)|^2)\,dy
\right)u^n_M.
\end{align}
We can then study the convergence properties of this scheme, beginning with the $(u,v)$-component.
\begin{prop}[Fully discrete convergence for the coupled $(u,v)$ system]\label{prop:fully_discrete_uv}
Let $k\ge 4$ be an integer. 
Let $(u(t),v(t))$ be the exact solution of \eqref{eqn:coupled_system} on $[0,T]$,
and assume
\[
\sup_{t\in[0,T]}\bigl(\|u(t)\|_{\Sigma^{k+4}}+\|v(t)\|_{\Sigma^{k+4}}\bigr)<\infty.
\]
Let $(u_M^n,v_M^n)$ be generated by 
\eqref{eqn:fully_discrete_uv_system}. Then there exist $\tau_0>0$, $M_0>0$ and a constant $C>0$ (depending only on
$T,k$ and $\sup_{t\in[0,T]}\bigl(\|u(t)\|_{\Sigma^{k+4}}+\|v(t)\|_{\Sigma^{k+4}}\bigr)$) such that for all $0<\tau<\tau_0$ and all $M\ge M_0$,
\begin{align}
\max_{0\le n\tau\le T}
\left(
\|u_M^n-u(n\tau)\|_{L^2}
+
\|v_M^n-v(n\tau)\|_{L^2}
\right)
\;\le\;
C\Bigl(\tau^2+ M^{\,1+\frac{1}{3}-\frac{k}{2}}\Bigr).
\label{eq:fully_discrete_uv_L2}
\end{align}
\end{prop}
To prove this estimate we will first consider the propagation of \eqref{eqn:fully_discrete_uv_system} using the modified initial conditions
\begin{align}\label{eqn:modified_initial_conditions12}
\tilde{u}^0_{M} &= \mathcal{Q}_{M}\!\left((E_{0})^2\,\psi_0\right), \quad
\tilde{v}^0_{M} = \mathcal{Q}_{M}\!\left(E_{0}\partial_x\bigl(E_{0}\psi_0\bigr)\right),
\end{align}
and then use a stability estimate to extend this to \eqref{eqn:true_fully_discrete_initial_cond1}-\eqref{eqn:true_fully_discrete_initial_cond3}.
\begin{lemma}\label{lem:approx_of_u_0_v_0} Let $u_M^0,v_M^0$ be given by \eqref{eqn:true_fully_discrete_initial_cond1}-\eqref{eqn:true_fully_discrete_initial_cond3} and $\tilde{u}_M^0,\tilde{v}_M^0$ be defined as in \eqref{eqn:modified_initial_conditions12}. Then, for every $k>\ell>\alpha>1/2$, there is a polynomial $F$ with coefficients that depend only on $\ell,\alpha,k,\delta$ such that for any $\psi_0\in \Sigma^{k+1}$ we have
\begin{align}\label{eqn:estimate_u_variable_initial_condition}
 \|\tilde{u}^0_M-u_M^0\|_{\Sigma^\ell}&\leq M^{1/3+\max\{\ell,\alpha\}/2-k/2}F(\|\psi_0\|_{\Sigma^k}),
 \\\label{eqn:estimate_v_variable_initial_condition}
 \|\tilde{v}^0_M-v_M^0\|_{\Sigma^\ell}&\leq M^{1/3+\max\{\ell,\alpha\}/2-k/2}F(\|\psi_0\|_{\Sigma^{k+1}}).
\end{align}
\end{lemma}
\begin{proof}
The proof of this statement is included in Appendix~\ref{app:proof_lemma_modified_initial_conditions}.
\end{proof}
\begin{proof}[Proof of Proposition~\ref{prop:fully_discrete_uv}] \georg{Consider $\tilde{u}^n_{M},\tilde{v}^n_{M}$ to be the approximations obtained via \eqref{eqn:fully_discrete_uv_system} with the initial conditions $\tilde{u}_M^0,\tilde{v}_M^0$ from \eqref{eqn:modified_initial_conditions12}. Then, analogously to the proof of Theorem~\ref{thm:fully_discrete_error_analysis_cubicNLSE}, we have
\begin{align*}
\max_{0\le n\tau\le T}
\left(
\|\tilde{u}_M^n-u(n\tau)\|_{L^2}
+
\|\tilde{v}_M^n-v(n\tau)\|_{L^2}
\right)
\;\le\;
C\Bigl(\tau^2+ M^{\,1+\frac{1}{3}-\frac{k}{2}}\Bigr),
\end{align*}
for a constant $C$ depending only on $T,k$ and $\sup_{t\in[0,T]}\bigl(\|u(t)\|_{\Sigma^{k+4}}+\|v(t)\|_{\Sigma^{k+4}}\bigr)$. The result then follows by a standard stability estimate analogous to Lemma~\ref{lem:stability_fully_discrete_NLSE} combined with Lemma~\ref{lem:approx_of_u_0_v_0}.
}
\end{proof}
This allows us to prove the following fully-discrete error bound on our novel R-transform-based integrator.
\begin{theorem}[Fully discrete convergence for DNLSE]\label{thm:fully-discrete_global_error_DNLSE}
Let $k\geq 4$ be an integer. We define the numerical reconstruction
\[
\psi_M^n(x)=\exp\!\left(-2i\delta \int_{-\infty}^x |u_M^n(y)|^2\,dy\right)u_M^n(x),
\]
and the exact reconstruction
\[
\psi(t_n,x)=\exp\!\left(-2i\delta \int_{-\infty}^x |u(t_n,y)|^2\,dy\right)u(t_n,x).
\]
Let $M_{k+5}=\sup_{t\in[0,T]}\|\psi(t)\|_{\Sigma^{k+5}}$. Then there exist $\tau_0>0$, $M_0>0$ and $C>0$ (depending only on $T,k,\delta$
and $M_{k+5}$) such that for all $0<\tau<\tau_0$ and $M\ge M_0$,
\begin{align}
\max_{0\le n\tau\le T}\|\psi_M^n-\psi(t_n)\|_{L^2}
\;\le\;
C\Bigl(\tau^2+ M^{\,1+\frac{1}{3}-\frac{k}{2}}\Bigr).
\label{eq:fully_discrete_psi_L2}
\end{align}
\end{theorem}

\begin{remark}
    Crucially, Theorem~\ref{thm:fully-discrete_global_error_DNLSE} does not impose any CFL-type constraints on $M,\tau$, meaning the scheme is truly unconditionally stable.
\end{remark}
\begin{proof}[Proof of Theorem~\ref{thm:fully-discrete_global_error_DNLSE}]
We proceed similarly to the proof of Theorem~\ref{thm:semi-discrete_global_error_DNLSE}:
\begin{align*}
    \|\psi_M^n-\psi(t_n)\|_{L^2}\leq \underbrace{\left\|e^{i\delta\int_{-\infty}^{x}\georg{\mathcal{Q}_M}(|u^n_{M}(y)|^2)-|u(t_n,y)|^2dy}u^n_{M}\right\|_{L^{2}}}_{\text{(A)}}+\underbrace{\|u^n_{M}-u(t_n)\|_{L^2}}_{\text{(B)}}
\end{align*}
Proposition~\ref{prop:fully_discrete_uv} ensures that\vspace{-0.5cm}
\begin{align*}
    \text{(B)}\leq C\left(\tau^2+M^{1+\frac{1}{3}-\frac{k}{2}}\right)
\end{align*}
for a constant $C>0$ depending on $M_{k+5}$ (using the relationship between $\|u\|_{\Sigma^k},\|v\|_{\Sigma^k}$ and $\|\psi\|_{\Sigma^{k+1}}$ established in the proof of Theorem~\ref{thm:semi-discrete_global_error_DNLSE}). \georg{For (A), similarly to \eqref{eqn:L2estimate_twist},} we combine the estimate \eqref{eqn:central_estimate_for_twisting_back}, Lemma~\ref{lem:Hermite_interpolation} and Proposition~\ref{prop:fully_discrete_uv} to show that also\vspace{-0.5cm}
\begin{align*}
    \text{(A)}\leq C\left(\tau^2+M^{1+\frac{1}{3}-\frac{k}{2}}\right)
\end{align*}
for a constant $C>0$ depending on $M_{k+5}$. This completes the proof.
\end{proof}
\section{Numerical examples}\label{sec:numerical_examples}
In the following we provide numerical examples which \revisions{support the theoretical analysis of \S\ref{sec:standard_cubic_NLSE} and \S\ref{sec:DNLSE} and illustrate the practical potential of Hermite spectral methods beyond harmonic traps in one, two and three spatial dimensions. In \S\ref{sec:examples_nlse_wo_potential} and \S\ref{sec:example_quintic_2d} we consider the cubic and quintic NLSE on $\mathbb{R}^2$, and demonstrate how the dispersive spreading of the solution degrades Fourier spectral methods on truncated domains, while the Hermite discretisation, being global on $\mathbb{R}^d$, is relatively unaffected. In \S\ref{sec:example_3d_potential} we solve a three-dimensional nonlinear Schr\"odinger equation with an anharmonic potential and compare cost against error. Finally, in \S\ref{sec:examples_dnlse} we demonstrate the unconditional stability and efficiency of our R-transform-based integrator for the DNLSE (cf. \S\ref{sec:stable_algorithm_DNLSE}) in comparison with a state-of-the-art Crank--Nicolson scheme.}\vspace{-0.2cm}
\subsection{NLSE without potential}\label{sec:examples_nlse_wo_potential}
As a first numerical example we look at a standard cubic NLSE \eqref{eqn:cubic_NLSE} \revisions{in 2d} with \revisions{$\mu=1$} and the initial profile 
\begin{align}\label{eqn:initial_conditions_numerics_DNLSE}
\revisions{\psi_0^{(2)}(x)=\psi_0(x_1)\psi_0(x_2),\quad \psi_0(x)
=
e^{i x - \frac{(x-1)^2}{2}}
+
e^{-\frac{(x+2)^2}{4}}.}
\end{align}
We consider the solution of this equation with a Fourier spectral method (i.e. \eqref{eqn:cubic_NLSE} posed on $[-L,L]^2$ with periodic boundary conditions, cf. e.g. \cite{lubich2008splitting}), and a Hermite spectral method as described in Section~\ref{sec:standard_cubic_NLSE}. We fix the number of spatial discretisation modes to $N=512^2$ in both methods and compare against a reference solution obtained with $N_{ref}=1024^2$ modes using the method from Section~\ref{sec:standard_cubic_NLSE}. The results for varying the window size $L$ of the Fourier spectral method can be seen in Figure~\ref{fig:Hermite_vs_Fourier_cubicNLSE}, the corresponding evolution of the solution is shown in Figure~\ref{fig:solution_plot_cubicNLSE}. We observe that the Hermite spectral method is, as shown in our analysis, unconditionally stable, and convergent at first order. The same is true, in principle, for the Fourier spectral method, with the caveat that the method approximates a periodised version of the original problem. This means that, as the mass of the solution moves out towards the boundary of the Fourier approximation window, the method incurs an approximation error (seen for example in Figure~\ref{fig:Hermite_vs_Fourier_cubicNLSE_T_3} for $L=10,20$). This can be overcome only by adjusting the width and location of the approximation window of the Fourier spectral method (cf. also \cite{iserles2024convergence} for an adaptive methodology for adjusting this approximation window). By contrast, the Hermite functions form an orthogonal basis on the whole real line, thus the approximation inherently captures the full solution features (so long as sufficiently many coefficients are retained).

\begin{figure}[h!]
		\centering
		\begin{subfigure}{0.5\textwidth}
			\centering
			\includegraphics[width=0.9\linewidth]{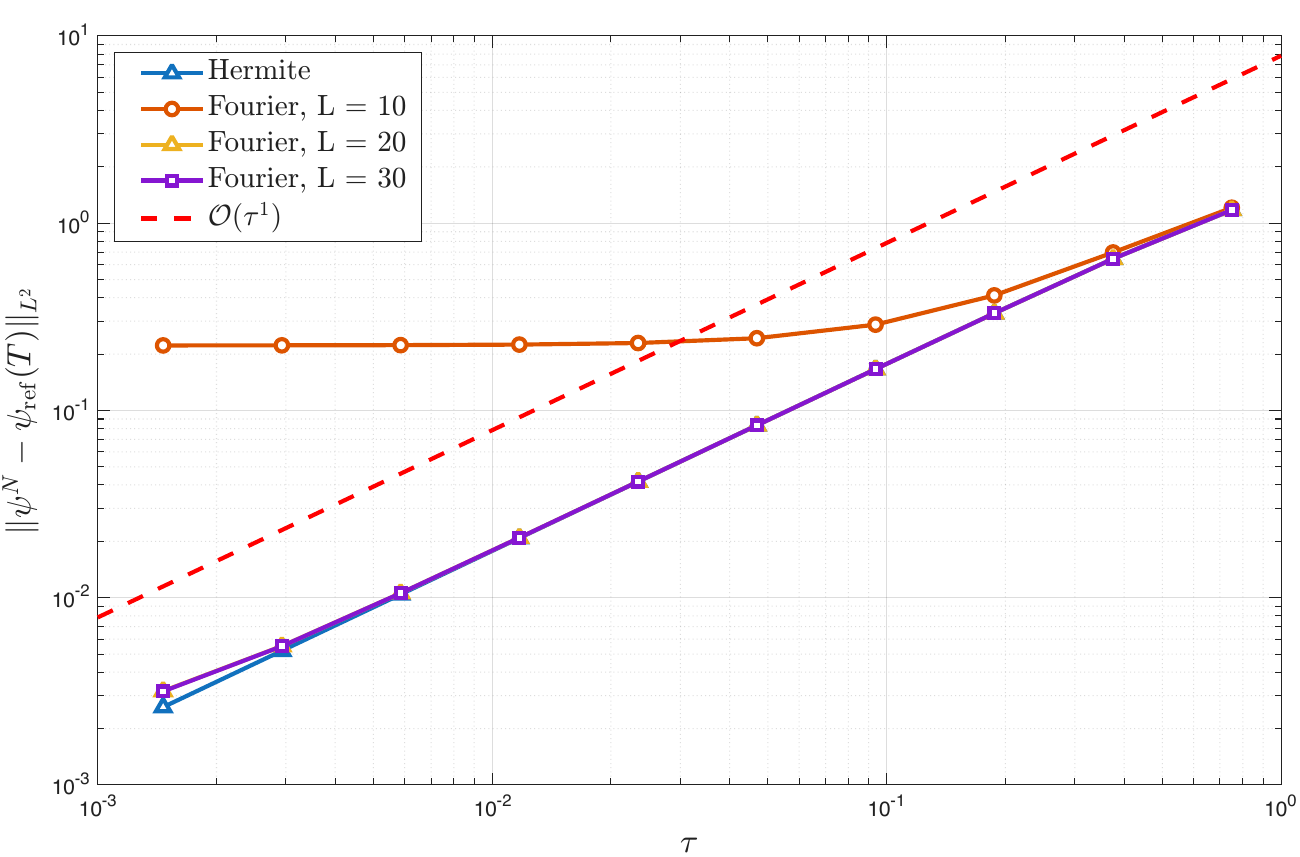}
			\caption{$T=1.5$.}\label{fig:Hermite_vs_Fourier_cubicNLSE_T_1-5}
		\end{subfigure}%
		\begin{subfigure}{0.5\textwidth}
		\centering
		\includegraphics[width=0.9\linewidth]{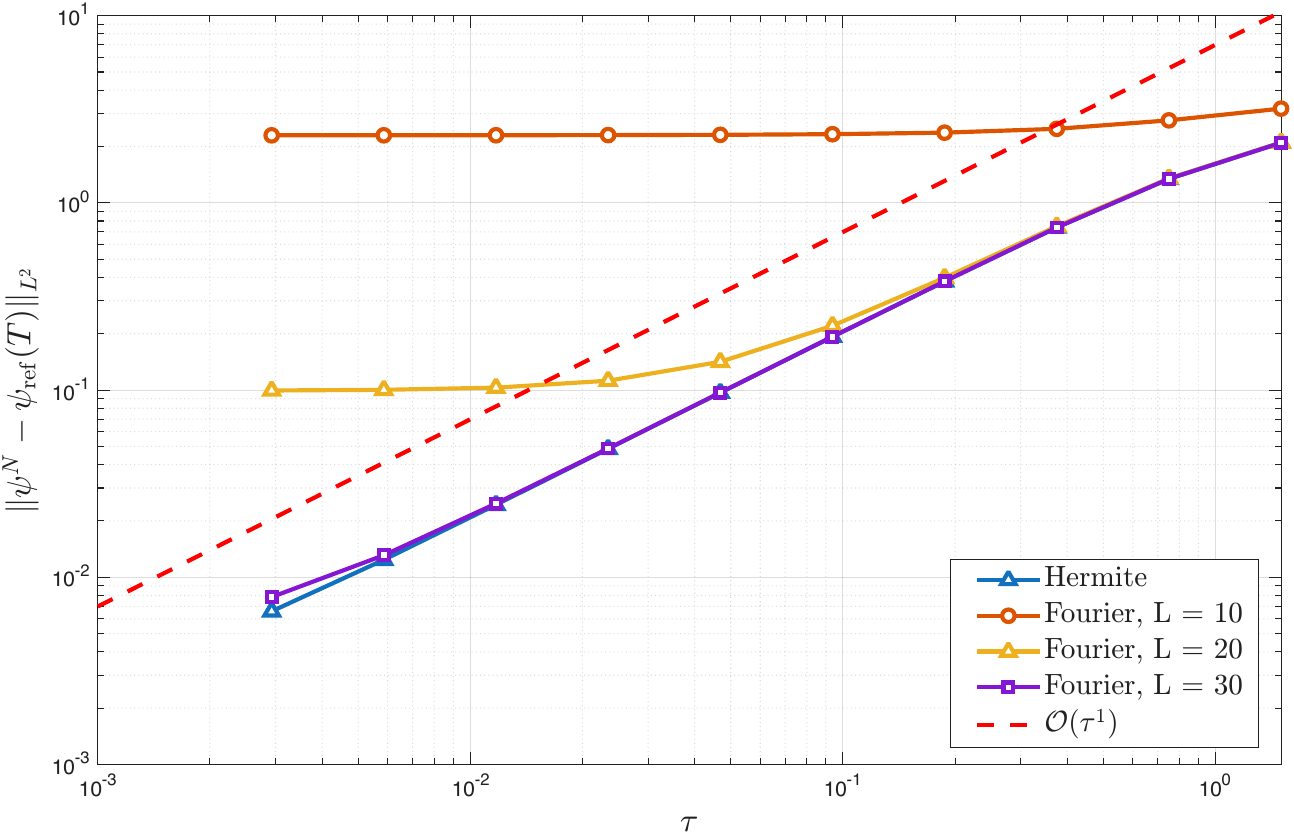}
		\caption{$T=3.0$.}\label{fig:Hermite_vs_Fourier_cubicNLSE_T_3}
		\end{subfigure}\vspace{-0.2cm}
		\caption{\revisions{Errors of the Fourier and Hermite spectral methods for \eqref{eqn:cubic_NLSE} with $N=512^2$ nodes, in 2D.}}
\label{fig:Hermite_vs_Fourier_cubicNLSE}
\end{figure}\vspace{-0.4cm}
\begin{figure}[h!]
		\centering
        \begin{subfigure}[t]{0.32\linewidth}
            \centering
            \includegraphics[width=0.9\linewidth]{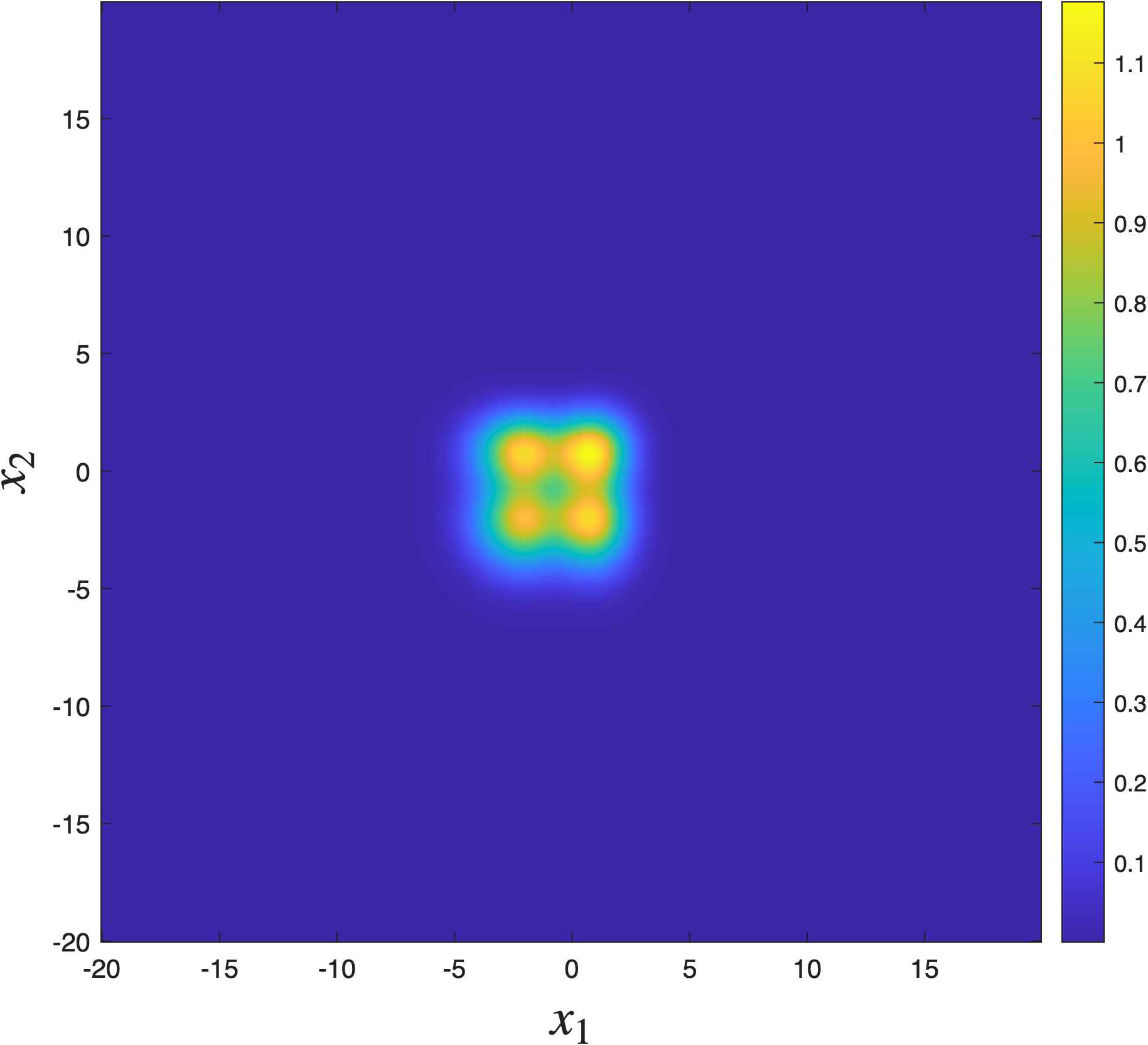}
            \caption{$T=0$.}
        \end{subfigure}
        \begin{subfigure}[t]{0.32\linewidth}
            \centering
            \includegraphics[width=0.9\linewidth]{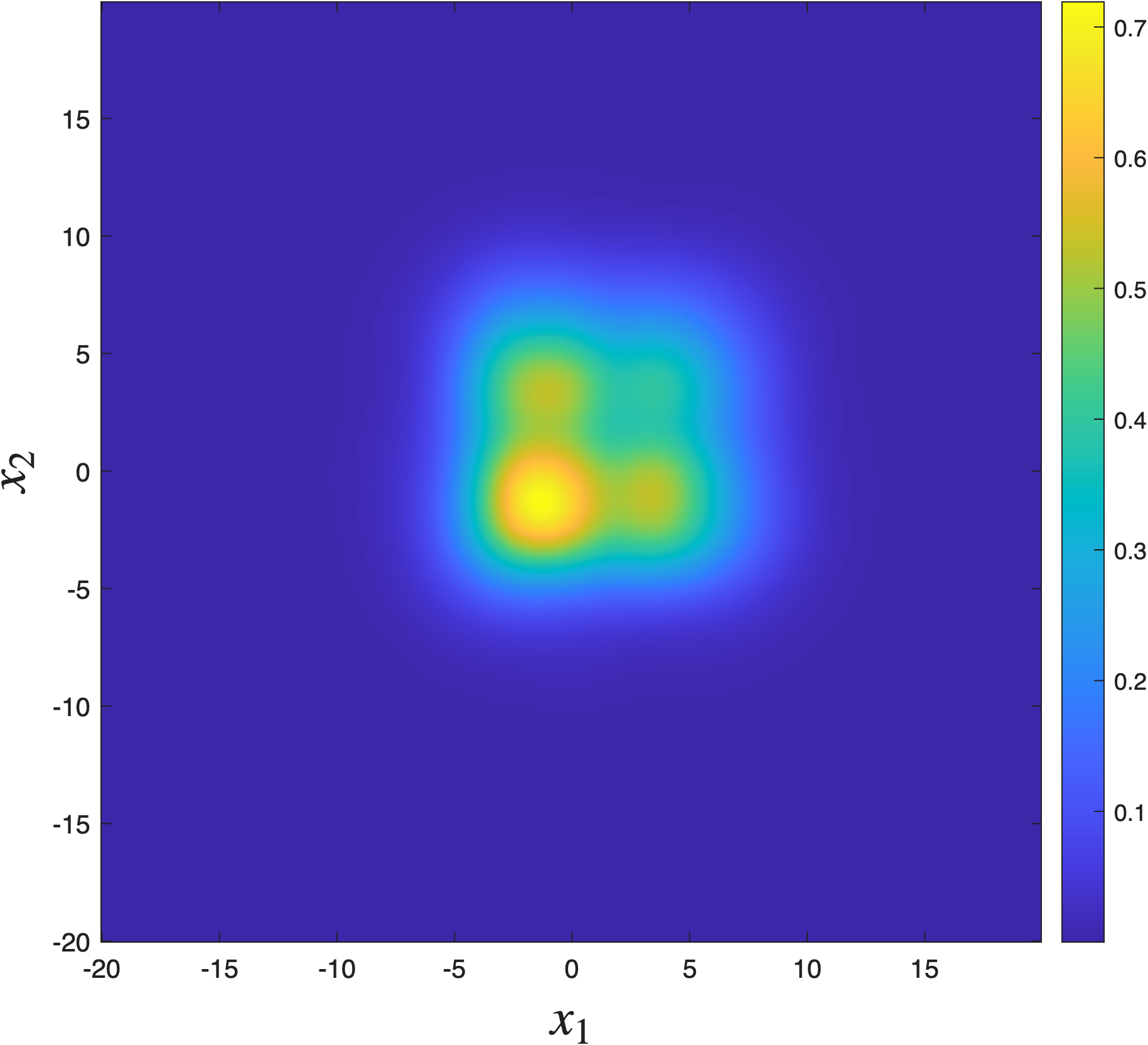}
            \caption{$T=1.5$.}
        \end{subfigure}
        \begin{subfigure}[t]{0.32\linewidth}
            \centering
            \includegraphics[width=0.9\linewidth]{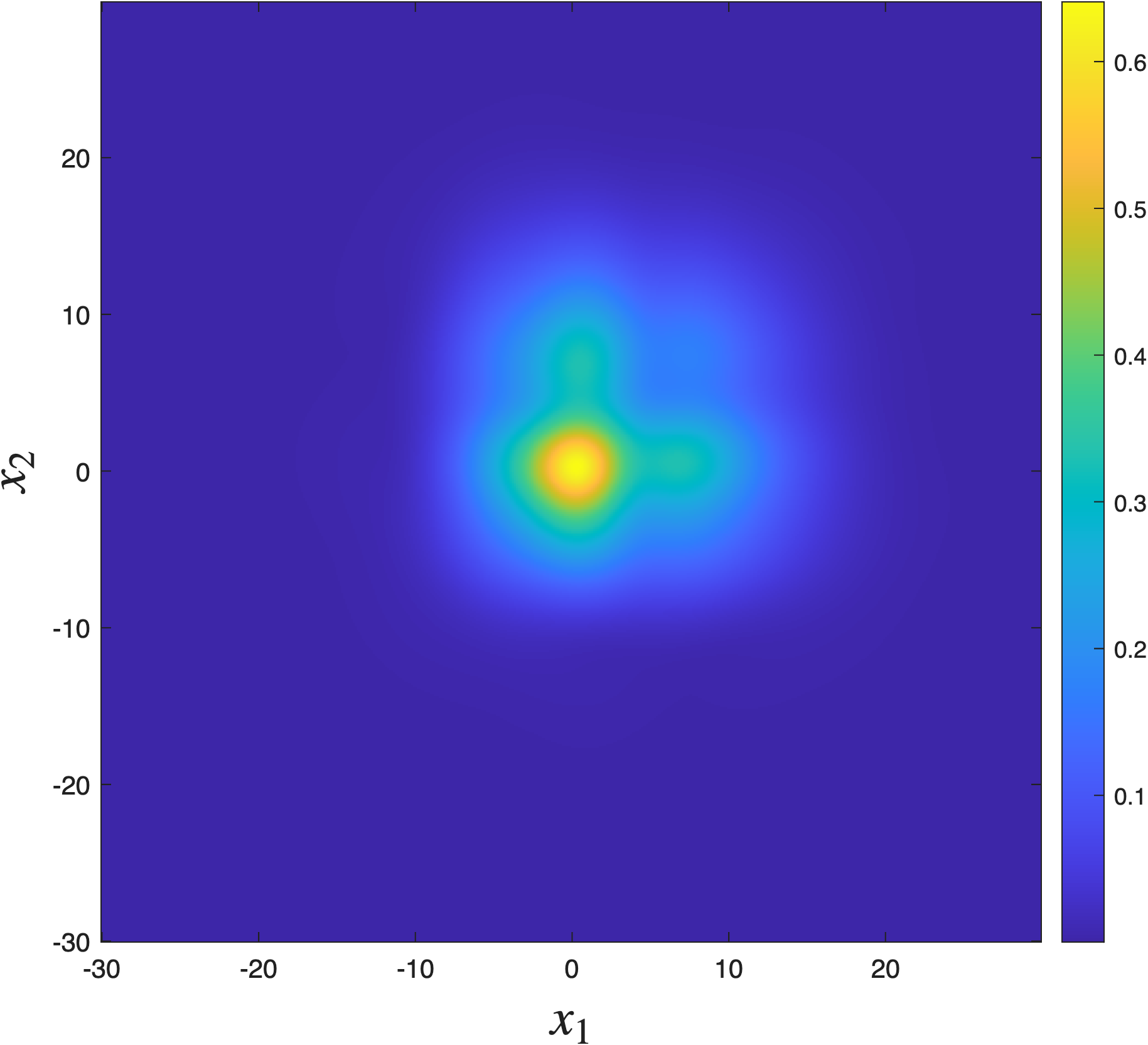}
            \caption{$T=3.0$.}
        \end{subfigure}\vspace{-0.2cm}
	\caption{\revisions{Solution field $|\psi|$ in 2D for the cubic NLSE \eqref{eqn:cubic_NLSE} with initial condition \eqref{eqn:initial_conditions_numerics_DNLSE}.}}\vspace{-1cm}
    \label{fig:solution_plot_cubicNLSE}
\end{figure}
\newpage\subsection{\revisions{Quintic NLSE in 2d}}\label{sec:example_quintic_2d}
\revisions{In this second example we consider the 2D quintic NLSE
\begin{align}\label{eqn:quintic_NLSE}
\begin{cases}	i\partial_t\psi = -\Delta\psi +|\psi|^4\psi,\\
\psi|_{t=0}=\psi^{(2)}_0.
\end{cases}
\end{align}}
\revisions{Quintic nonlinearities arise, for instance, in mean-field models of degenerate Bose gases in which three-body interactions are significant \cite{chen2011quintic}.} \revisions{We note that the convergence results derived in S\ref{sec:standard_cubic_NLSE} use the algebra property of $\Sigma^k$ and can therefore easily be extended to the quintic case with similar regularity assumptions, in the interest of brevity the full analysis is omitted.} \revisions{Our numerical setup in this case is similar to \S\ref{sec:examples_nlse_wo_potential}, taking the 2D initial data $\psi^{(2)}_0(x)=\psi_0(x_1)\psi_0(x_2)$ with $\psi_0$ as specified in \eqref{eqn:initial_conditions_numerics_DNLSE}. We compare the solution of this equation using the Lie splitting \eqref{eqn:Lie_splitting} combined with a Fourier spectral method posed on $[-L,L]^2$ with periodic boundary conditions and a Hermite spectral method with the natural tensor product extension from \S\ref{sec:spatial_truncation_error} respectively. We fix the number of spatial discretisation modes to $N=512^2$ in both methods and compare against a reference solution obtained using the Hermite spectral method with $N_{ref}=1024^2$ modes and timestep $\tau_{ref}=10^{-4}$. The evolution of the solution field can be seen in Figure~\ref{fig:solution_fields_quintic_nlse}, and the error for varying window-sizes $L$ can be seen in Figure~\ref{fig:quintic_nlse_convergence}. We observe similar behaviour to the previous cubic NLSE example: the Hermite spectral method is unconditionally stable and convergent at first order. In principle the same holds for the Fourier spectral method, however the artificial truncation of the domain means again that the spatial approximation error is significant when the mass of the solution approaches the boundary of the window.}
\begin{figure}[h!]
		\centering
		\begin{subfigure}{0.5\textwidth}
			\centering
			\includegraphics[width=0.99\linewidth]{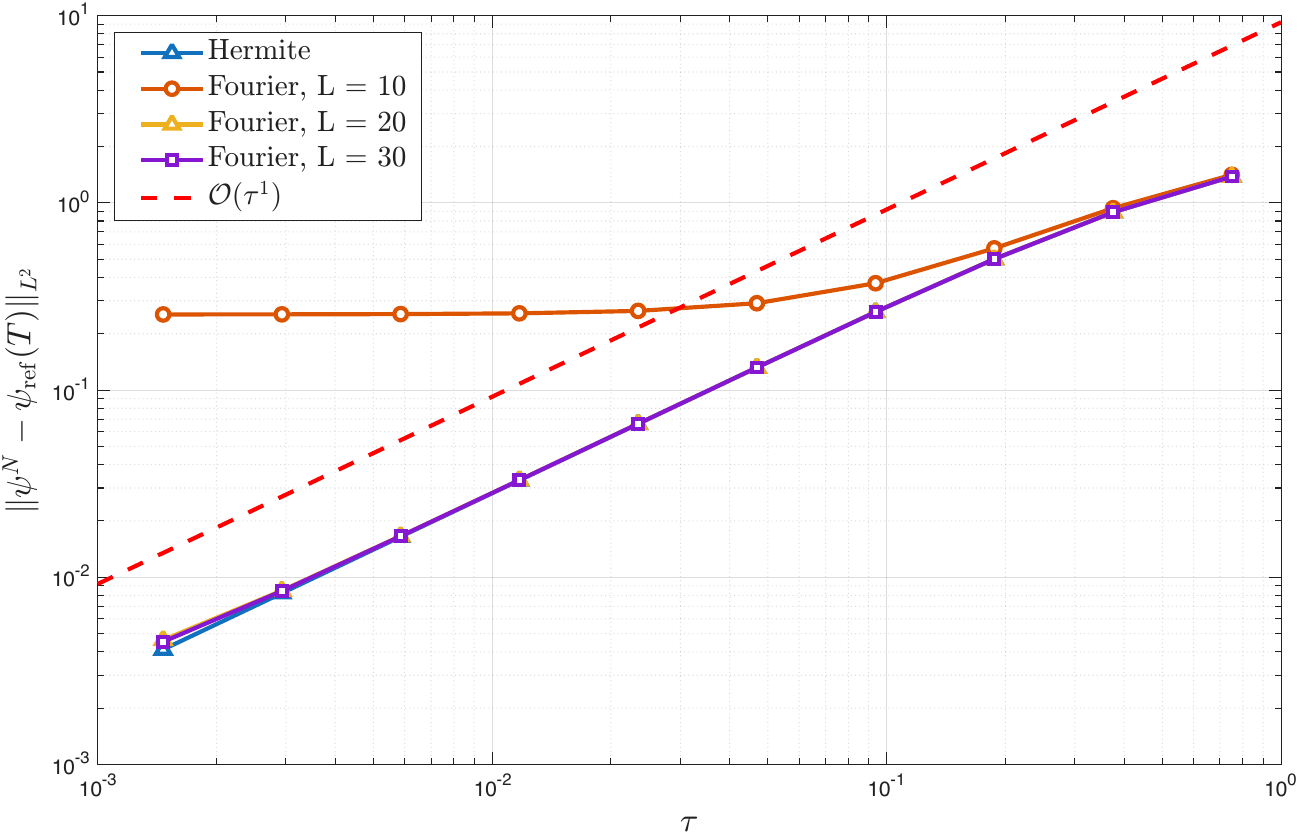}
			\caption{$T=1.5$.}
		\end{subfigure}%
		\begin{subfigure}{0.5\textwidth}
		\centering
		\includegraphics[width=0.99\linewidth]{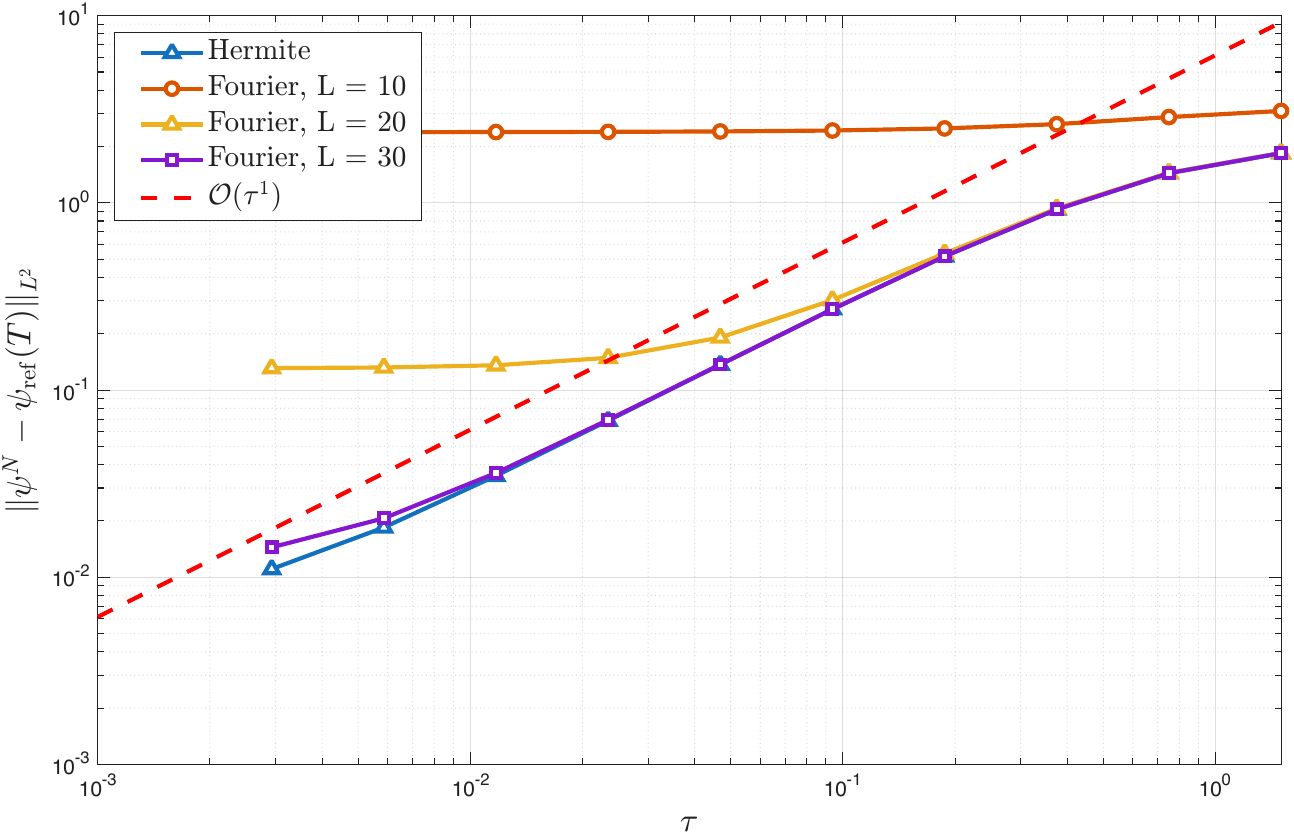}
		\caption{$T=3.0$.}
		\end{subfigure}
		\caption{\revisions{Errors of the Fourier and Hermite spectral methods for \eqref{eqn:quintic_NLSE} with $N=512^2$ nodes, in 2D.}}
        \label{fig:quintic_nlse_convergence}
\end{figure}
\begin{figure}[h!]
		\centering
        \begin{subfigure}[t]{0.32\linewidth}
            \centering
            \includegraphics[width=\linewidth]{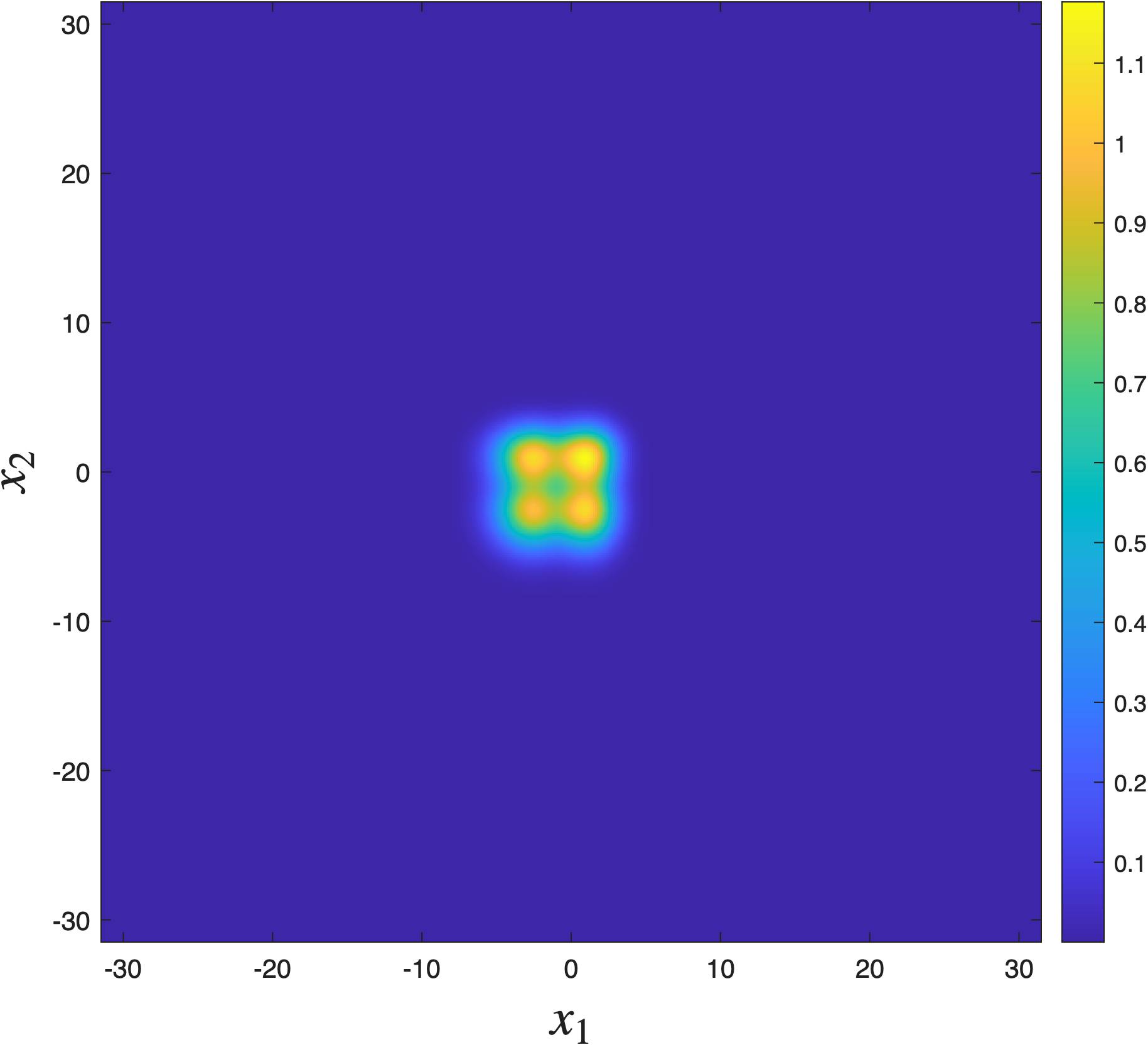}
            \caption{$T=0$.}
        \end{subfigure}
        \begin{subfigure}[t]{0.32\linewidth}
            \centering
            \includegraphics[width=\linewidth]{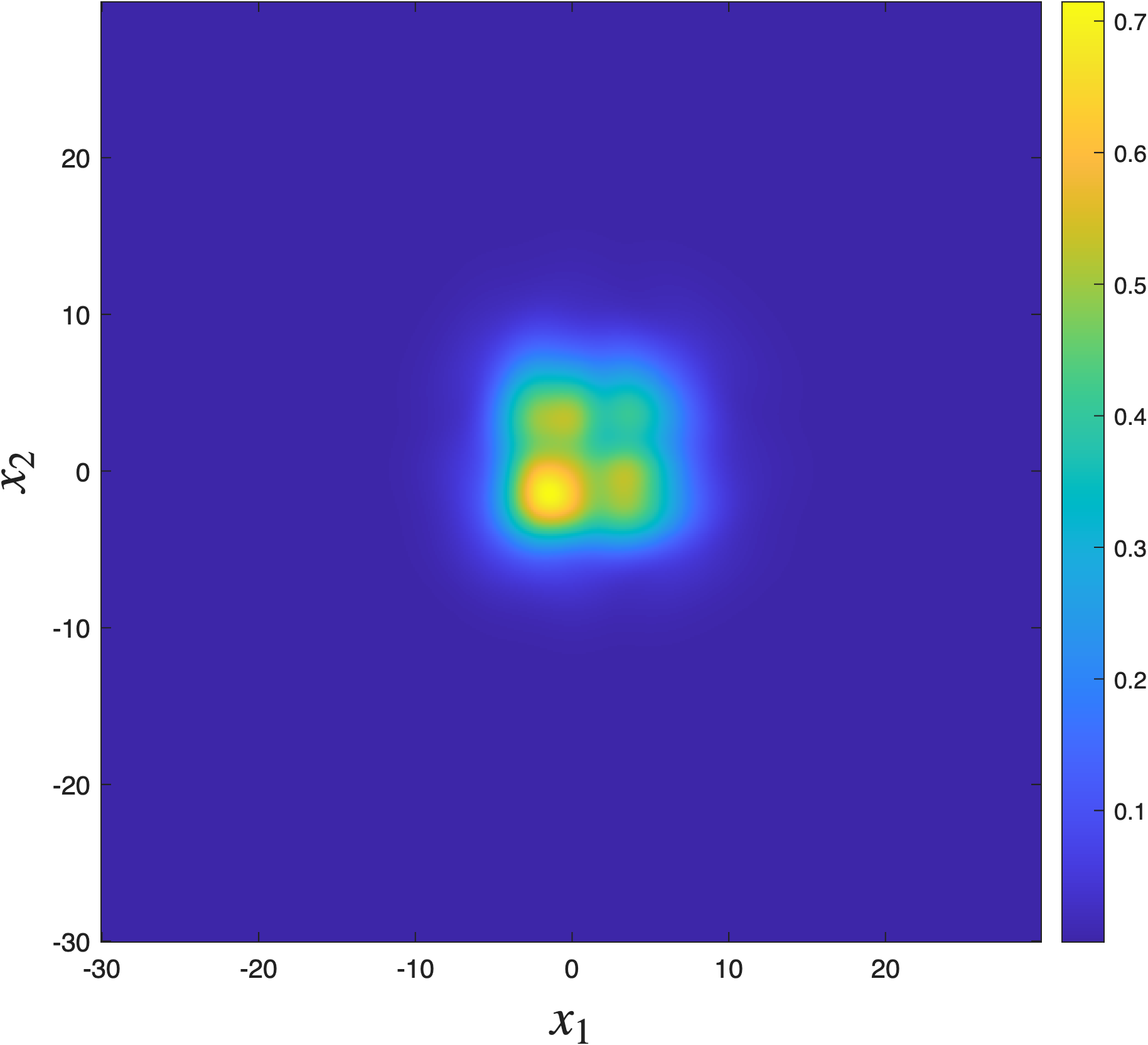}
            \caption{$T=1.5$.}
        \end{subfigure}
        \begin{subfigure}[t]{0.32\linewidth}
            \centering
            \includegraphics[width=\linewidth]{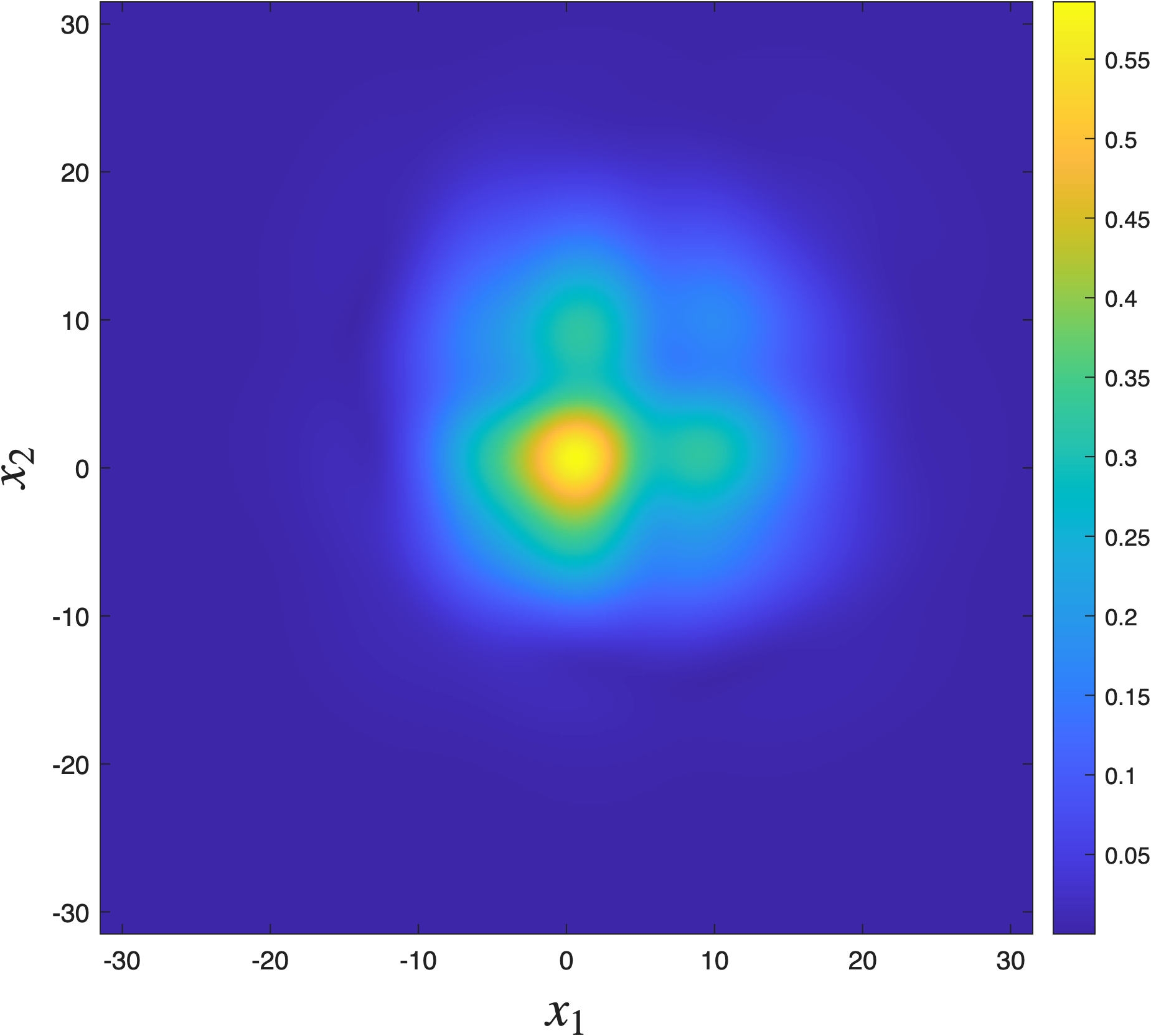}
            \caption{$T=3.0$.}
        \end{subfigure}
	\caption{\revisions{Solution field $|\psi|$ in 2D for the quintic NLSE \eqref{eqn:quintic_NLSE}.}}
    \label{fig:solution_fields_quintic_nlse}
\end{figure}

\newpage\subsection{\revisions{Nonlinear Schr\"odinger equation with a quartic potential in 3d}}\label{sec:example_3d_potential}

\revisions{In this third example we consider a Schr\"odinger equation with a purely
quartic (and hence strongly confining, but not harmonic) potential,
\begin{align}\label{eqn:3d_schroedinger}
   \begin{cases}
   i\partial_t\psi = -\Delta\psi + V(x)\psi - \mu |\psi|^2 \psi,\\
    \psi|_{t=0}=\psi^{(3)}_0,
   \end{cases}
\end{align}
where $V(x) = b|x|^4$ with $b=0.05$ and $\mu=1$, posed on $\mathbb{R}^3$. As initial data we take a
Gaussian wave packet displaced from the centre of the domain,
\begin{align*}
\psi^{(3)}_0(x) = \exp\!\left(-\tfrac{1}{2}\left((x_1-d_0)^2+x_2^2+x_3^2\right)\right),
\qquad d_0 = 3,
\end{align*}
and we compare the
performance of both types of spectral methods by integrating up to $T=4.0$. We
note that this case is, in principle, not directly covered by our analysis in
\S\ref{sec:standard_cubic_NLSE} since the potential is unbounded, however one
can use techniques similar to \cite{neuhauser2009convergence} (cf. also
\cite{iserles2024splitting}) to prove convergence for suitably regular initial
data.} \revisions{Both spatial discretisations are combined with the Lie splitting
\eqref{eqn:Lie_splitting}.}

\revisions{A fair comparison with Fourier spectral methods should account for
the cost of the underlying transforms: the tensor-product Hermite transform in
our implementation \cite{maierhoferwebb25} costs $\mathcal{O}(N^{d+1})$
operations for $N$ modes per spatial dimension (a dense
$N\times N$ matrix applied along each of the $d$ dimensions), compared with
$\mathcal{O}(N^d\log N)$ for the FFT. This gap can be narrowed
using fast Hermite transforms (cf. e.g. \cite{LEIBON2008211}). In this example
we demonstrate that, even with the standard Hermite transform, improved accuracy per degree of freedom can outweigh the higher per-mode cost over the Fourier method. In particular, this example illustrates that the choice of truncation parameter $L$ in the Fourier method is critical to its performance: if the parameter is too small, the solution is truncated and the error stagnates; if it is too large, the number of modes required to resolve the solution increases and the cost rises. By contrast, the Hermite method automatically resolves the region of phase space occupied by the solution so long as the number of spatial modes is sufficiently large, without having to fine-tune an additional parameter.}

\revisions{The results of this comparison can be seen in Figure~\ref{fig:cost_error_3d}, where we
compare the two methods for a range of spatial discretisation parameters $N$
and domain sizes $L\in\{5,10,20,30\}$. We measure the error in
$L^2(\mathbb{R}^3)$ at time $T=4.0$ for a fixed timestep
$\tau = 10^{-3}$ against a reference solution obtained
with the Hermite spectral method using $N_{\mathrm{ref}}=96^3$ spatial modes
and the second-order Strang splitting with timestep
$\tau_{\mathrm{ref}}=10^{-4}$. We observe that the optimal window choice appears to be around $L=10$, which matches the cost-for-accuracy of the Hermite spectral method.}

\begin{figure}[h!]
		\centering
		\begin{subfigure}{0.5\textwidth}
			\centering
			\includegraphics[width=0.99\linewidth]{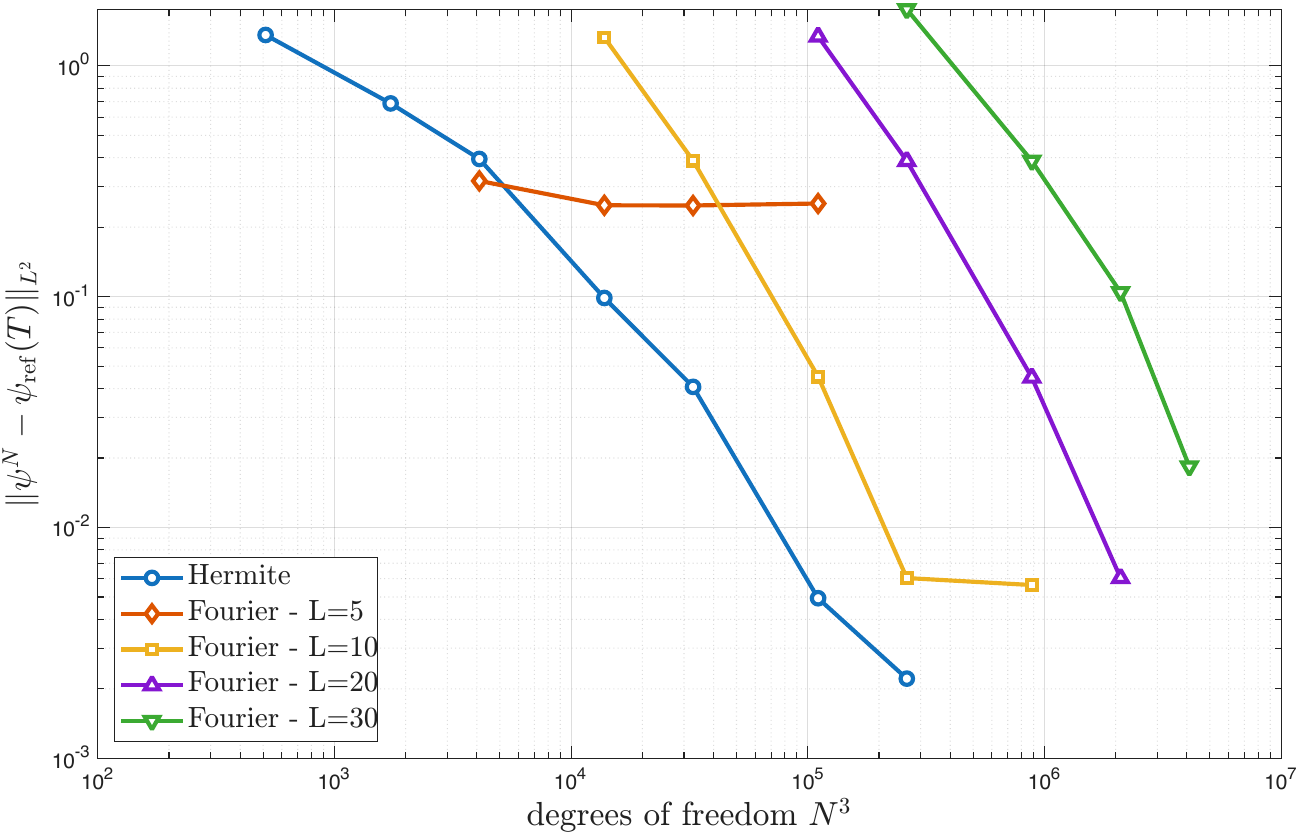}
			\caption{Error vs degrees of freedom.}
		\end{subfigure}%
		\begin{subfigure}{0.5\textwidth}
		\centering
		\includegraphics[width=0.99\linewidth]{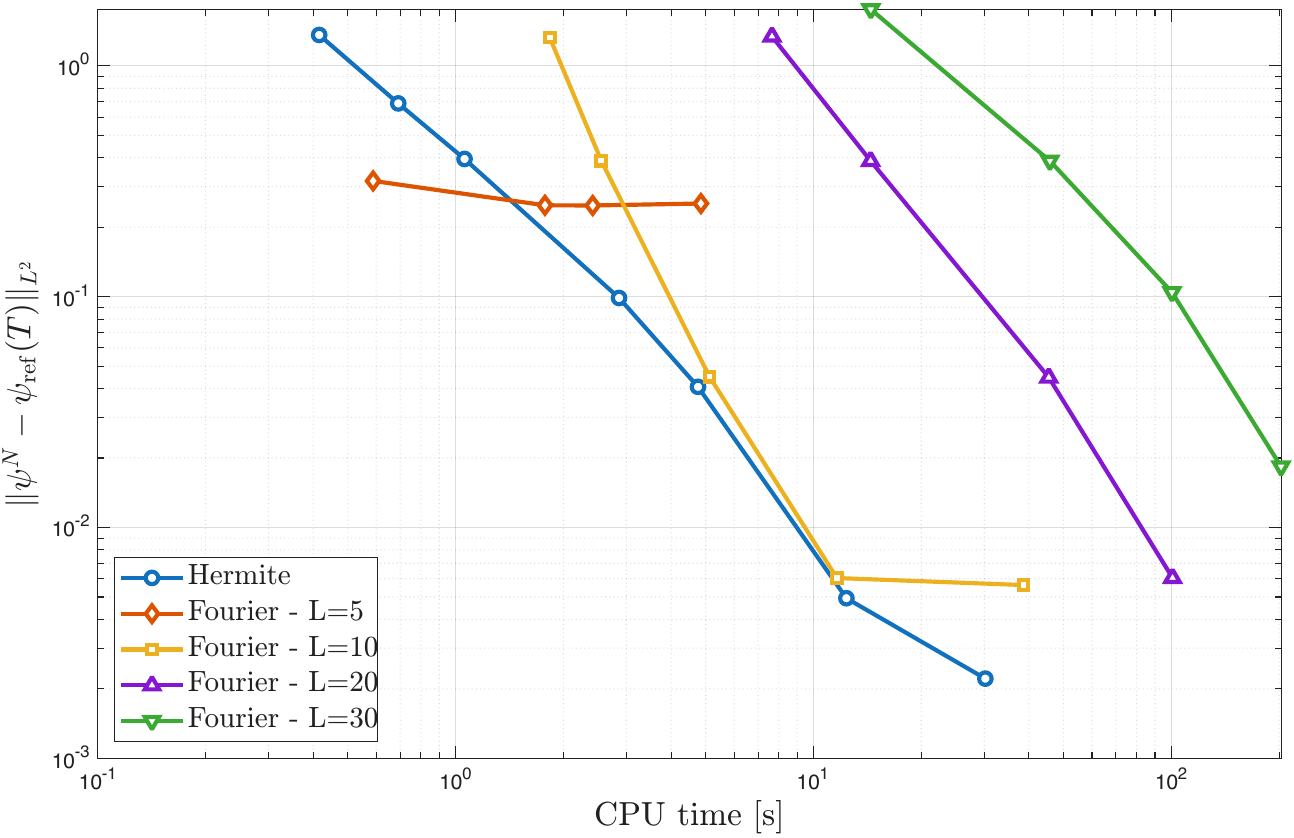}
		\caption{Error vs CPU time.}
		\end{subfigure}
		\caption{Cost vs. error comparison of Hermite and Fourier spectral methods for \eqref{eqn:3d_schroedinger}.}
		\label{fig:cost_error_3d}
\end{figure}
\subsection{Simulating the DNLSE}\label{sec:examples_dnlse}

Finally, we consider the solution of \eqref{eqn:DNLSE} with the initial condition \eqref{eqn:initial_conditions_numerics_DNLSE} and $\delta=1$. In the following experiments we used a direct Runge--Kutta method (RK4) to compute a reference solution using $M_{ref}=1000$ Hermite modes and $\tau_{ref}=10^{-4}$. We evaluate our new scheme presented in \S\ref{sec:stable_algorithm_DNLSE} (denoted by ``R-transform'' in the graphs) against the state-of-the-art Crank--Nicolson scheme presented in \cite{XUE2024134372}, denoted by ``Direct'' in the following graphs. The evolution of the solution profile can be seen in Figure~\ref{fig:evolution_solution_profile}. We note that the Crank--Nicolson scheme leads to instability after short simulation times, while our novel method is able to resolve the solution dynamics stably for longer time intervals (and without CFL-constraints).

\begin{figure}[h!]
    \centering
    \begin{subfigure}{0.5\textwidth}
        \centering
        \includegraphics[width=0.99\linewidth]{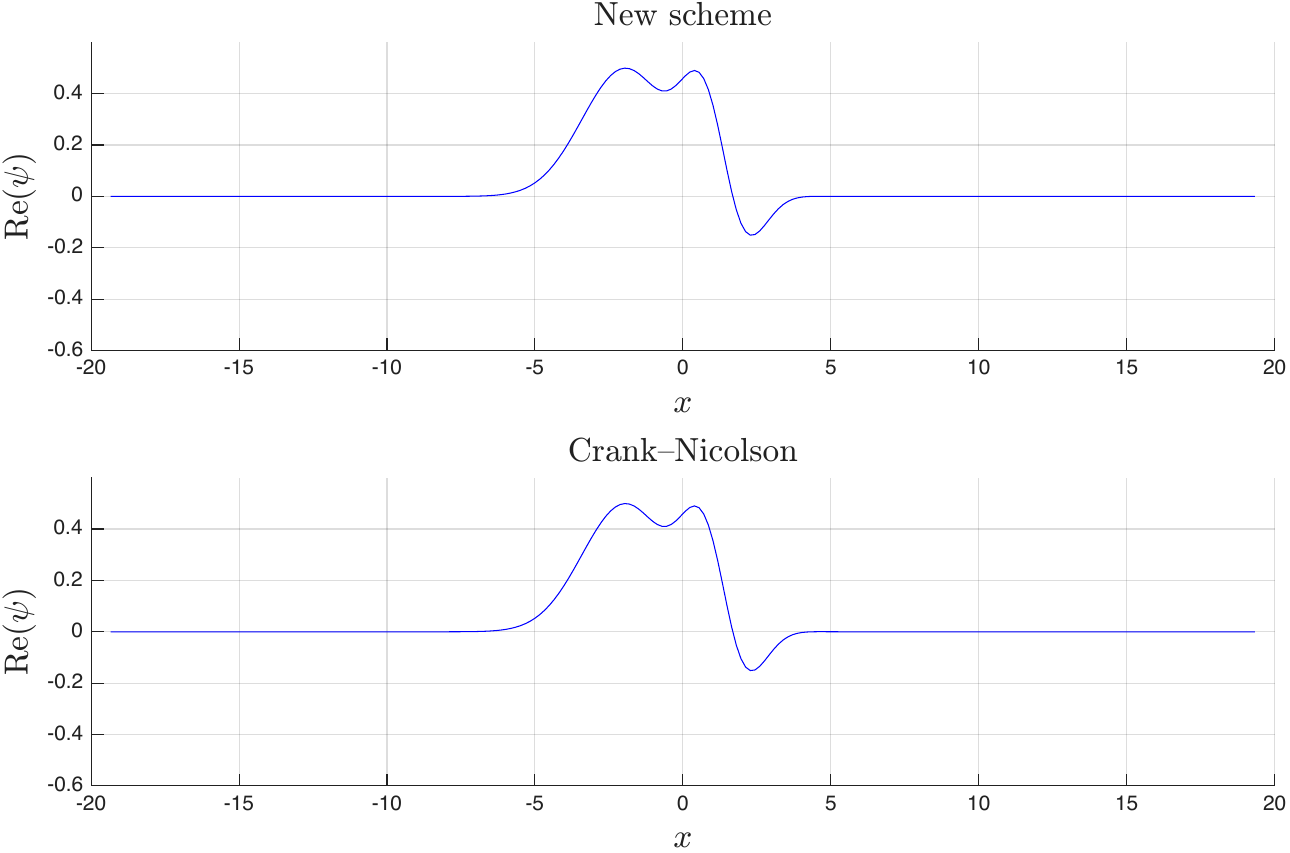}
        \caption{$t=0$.}
    \end{subfigure}%
    \begin{subfigure}{0.5\textwidth}
    \centering
    \includegraphics[width=0.99\linewidth]{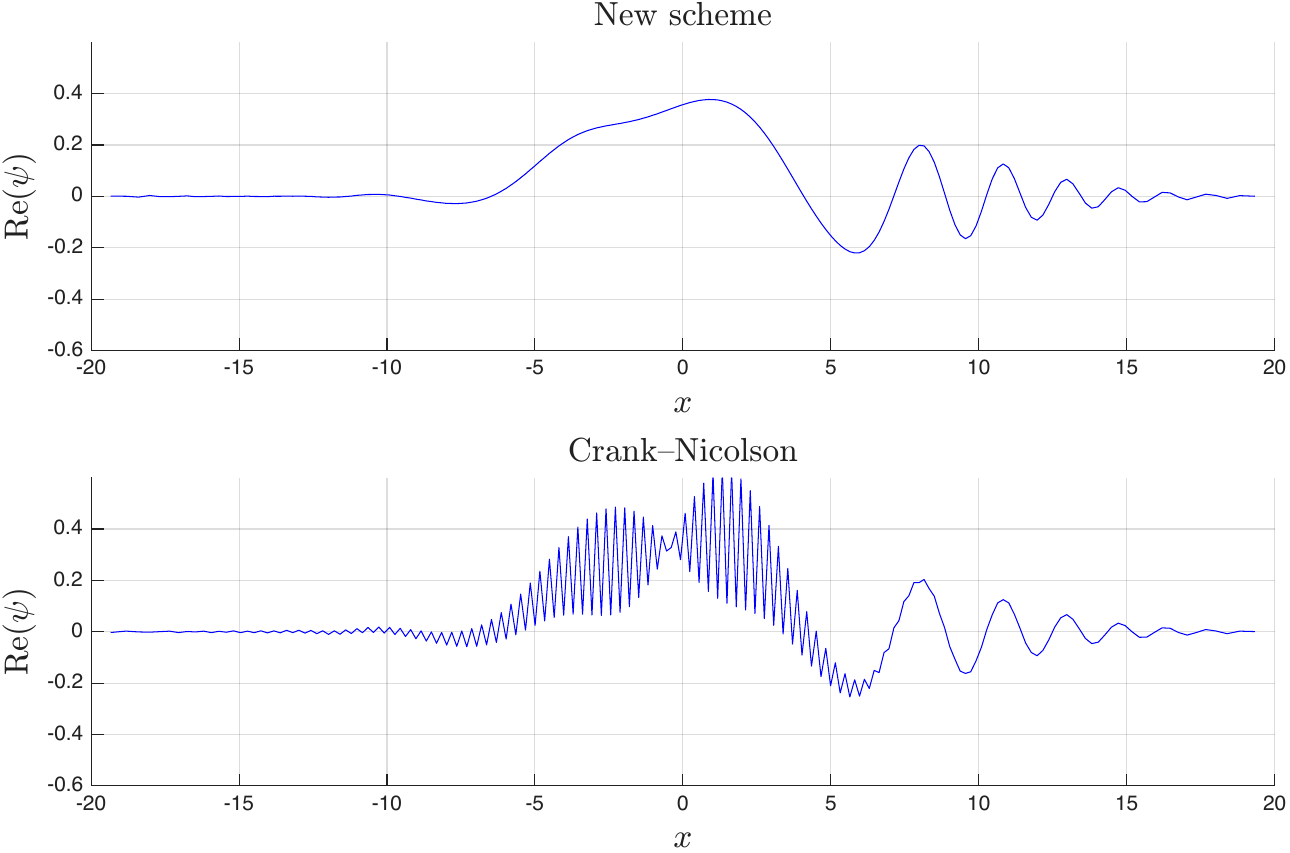}
       \caption{$t=1.8375$.}
    \end{subfigure}
    \caption{Evolution of the solution profile \eqref{eqn:initial_conditions_numerics_DNLSE} with $M=200,\tau=0.0075$.}
    \label{fig:evolution_solution_profile}
\end{figure}

To evaluate the numerical performance of our novel method more thoroughly we can now consider the following convergence plots, Figures~\ref{fig:error_200_0-1}-\ref{fig:error_500_1-0}. We note that the Crank--Nicolson method requires the CFL condition $\tau\lesssim M^{-2}$ for convergence, and at those places where no blue curve with circular markers is present the CN method diverged. From the three graphs the clear advantage of the unconditional stability in our new scheme is apparent. In particular, as the solution progresses over longer times, dispersive effects mean that a larger number of Hermite modes are required to accurately resolve the solution behaviour. This is perfectly fine to do in the unconditionally stable new ``R-transform''-based method, but causes significant problems in the CN method as can be observed in Figure~\ref{fig:error_500_1-0}. An additional marked advantage of the new methodology is that it is fully explicit thus leading to a much more efficient numerical scheme.
\begin{figure}[h!]
    \centering
    \begin{subfigure}{0.5\textwidth}
        \centering
        \includegraphics[width=0.93\linewidth]{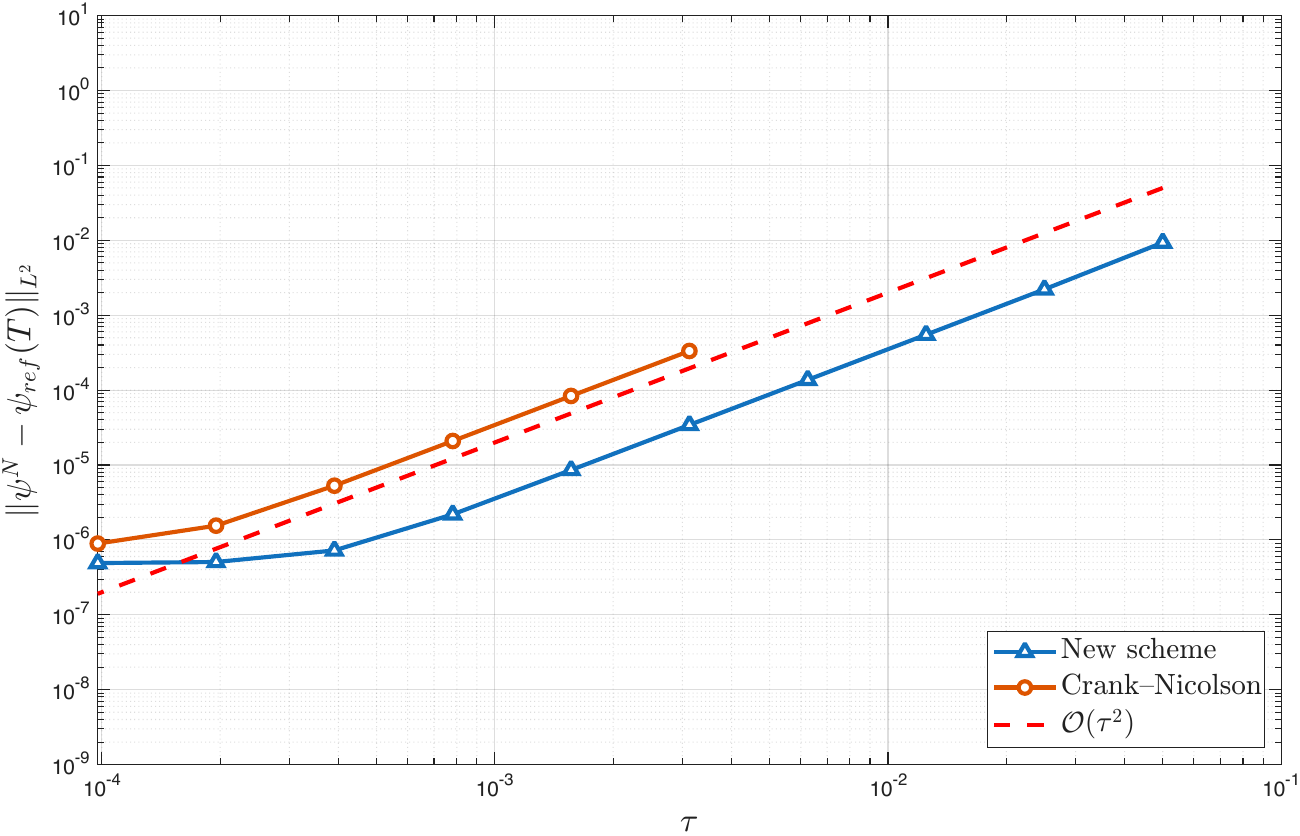}
        \caption{Convergence order.}
        \label{fig:local_error}
    \end{subfigure}%
    \begin{subfigure}{0.5\textwidth}
    \centering
    \includegraphics[width=0.93\linewidth]{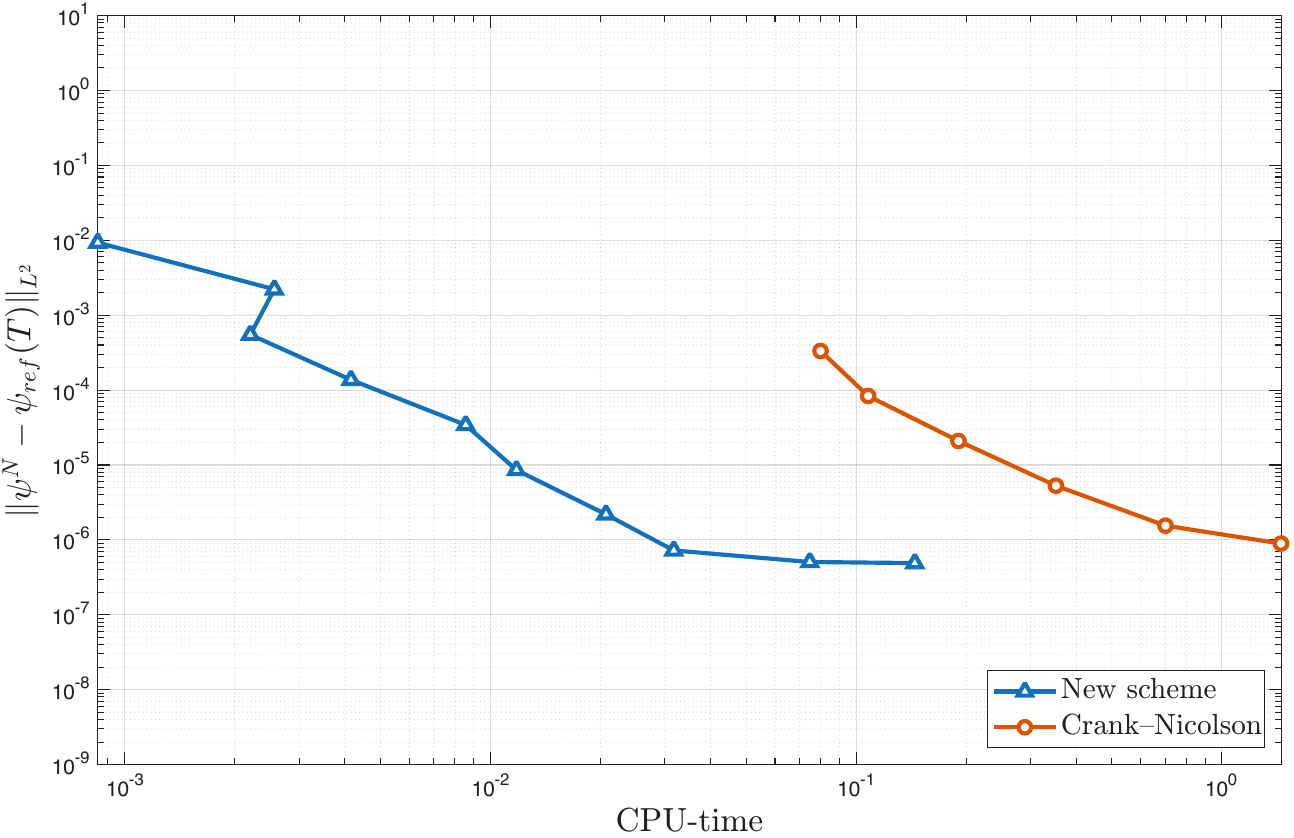}
    \caption{Error vs CPU time.}
    \label{fig:global_error}
    \end{subfigure}
    \caption{Global error of the numerical solution for $T=0.1, N=200$.}
    \label{fig:error_200_0-1}
\end{figure}

\begin{figure}[h!]
    \centering
    \begin{subfigure}{0.5\textwidth}
        \centering
        \includegraphics[width=0.93\linewidth]{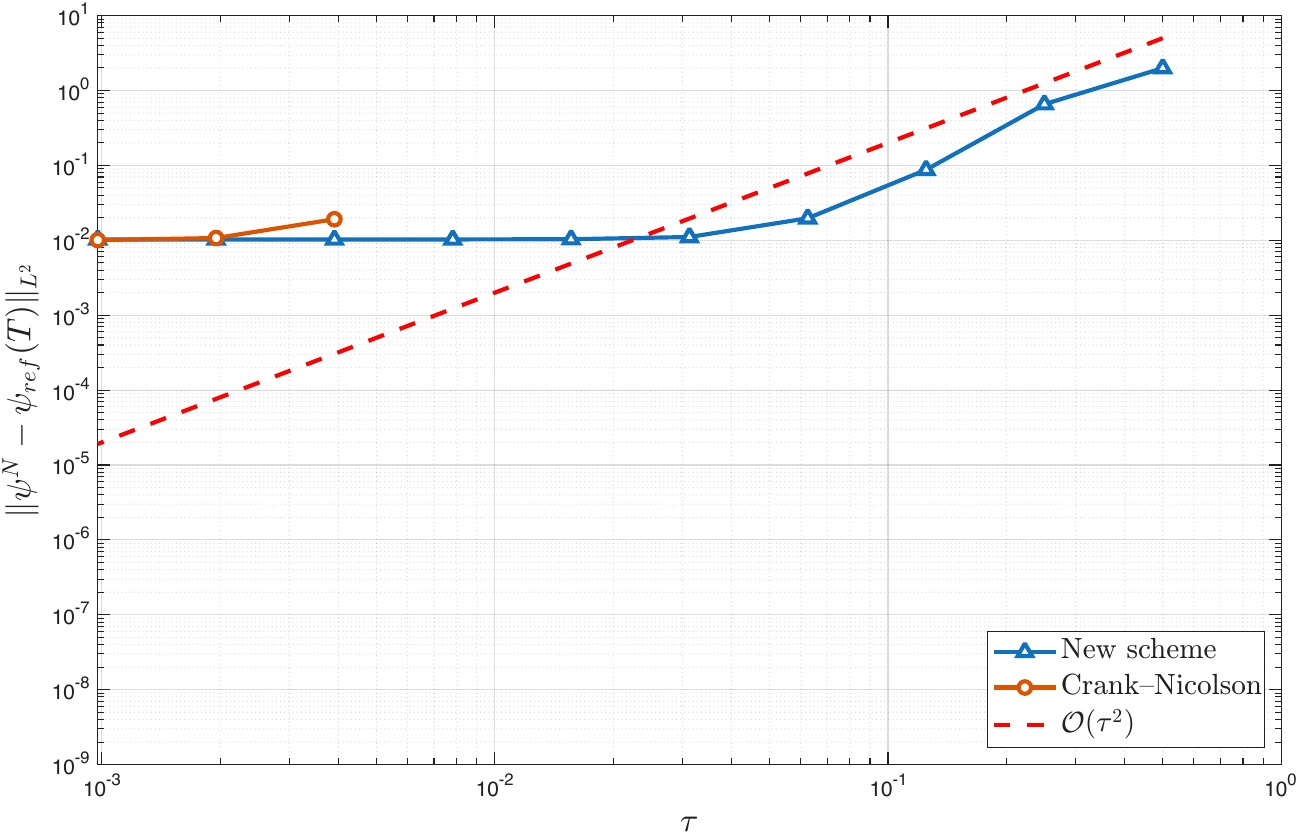}
        \caption{Convergence order.}
        \label{fig:local_error}
    \end{subfigure}%
    \begin{subfigure}{0.5\textwidth}
    \centering
    \includegraphics[width=0.93\linewidth]{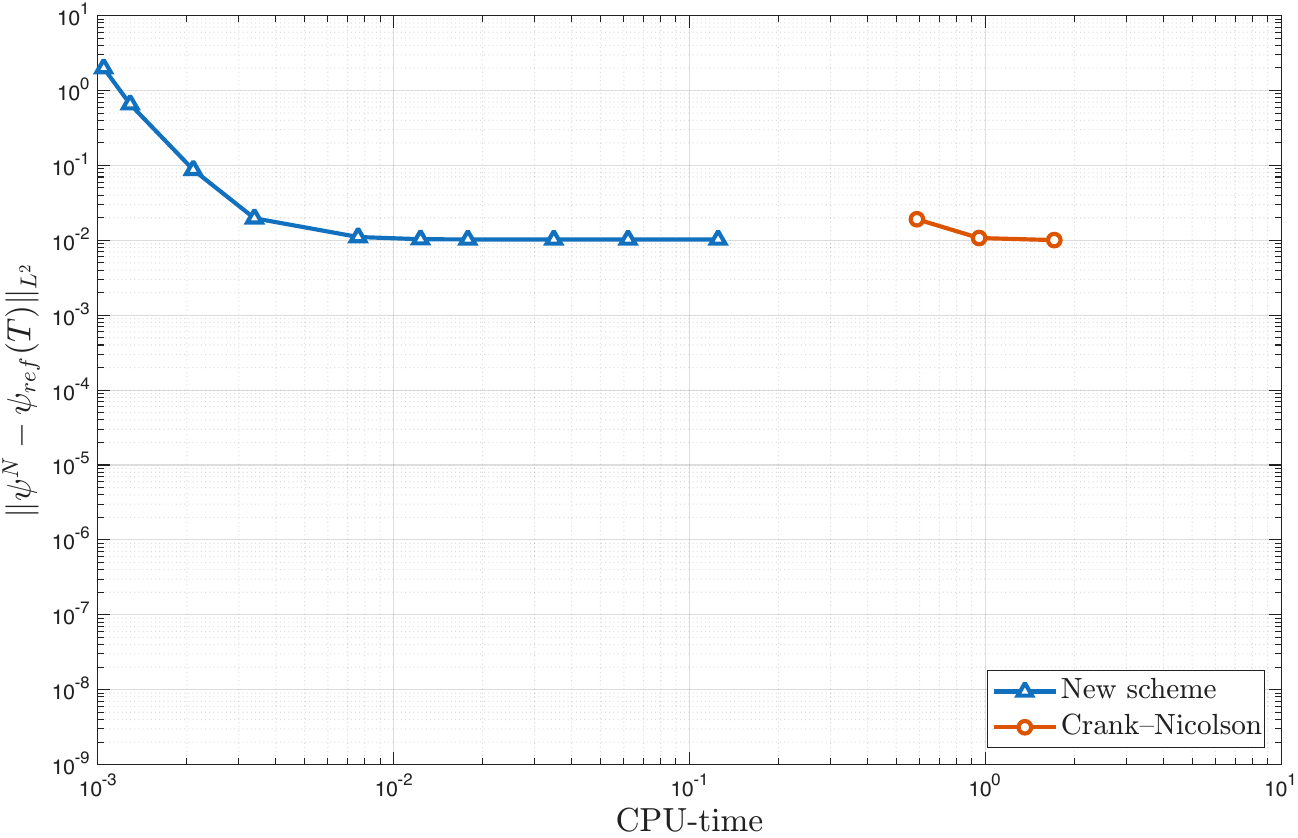}
    \caption{Error vs CPU time.}
    \label{fig:global_error}
    \end{subfigure}
    \caption{Global error of the numerical solution for $T=1.0, N=200$.}
    \label{fig:error_200_1-0}
\end{figure}

\begin{figure}[h!]
    \centering
    \begin{subfigure}{0.5\textwidth}
        \centering
        \includegraphics[width=0.93\linewidth]{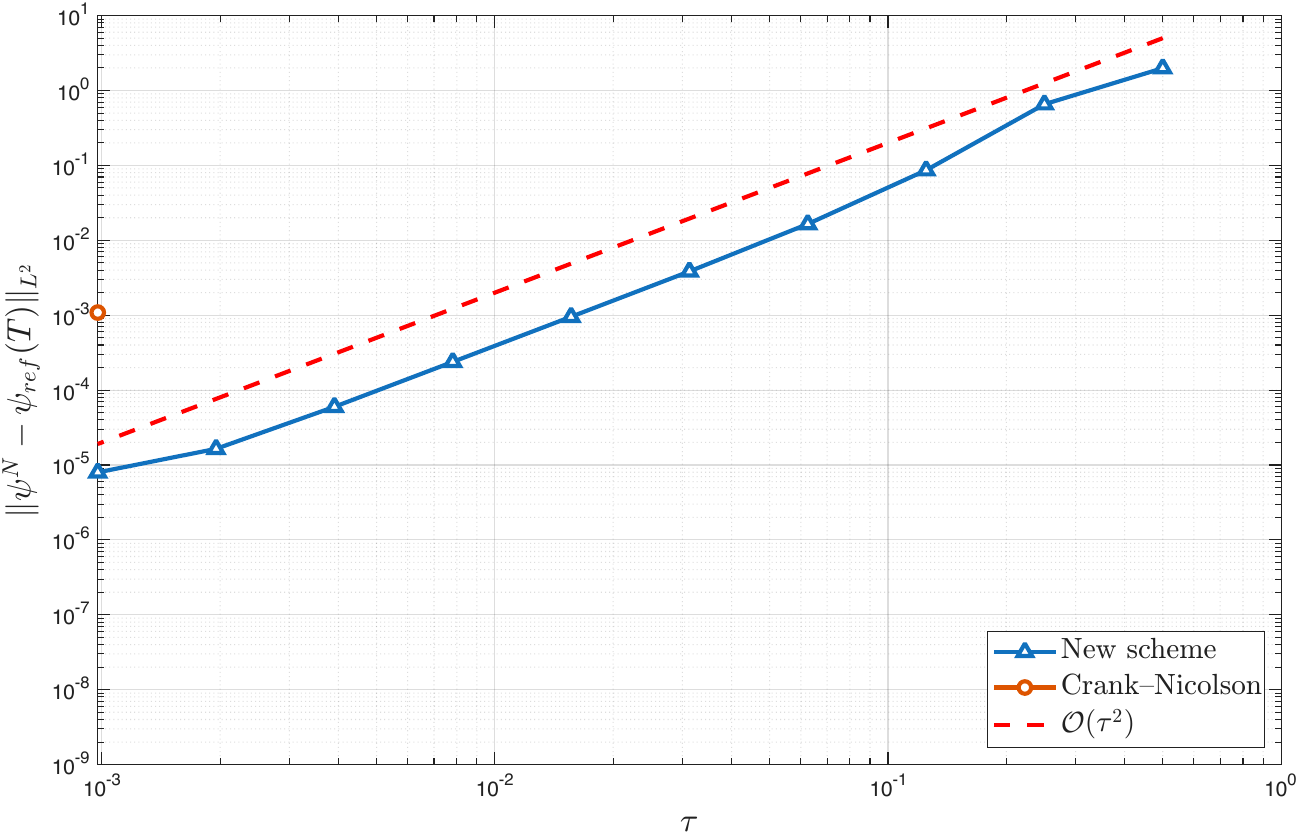}
        \caption{Convergence order.}
        \label{fig:local_error}
    \end{subfigure}%
    \begin{subfigure}{0.5\textwidth}
    \centering
    \includegraphics[width=0.93\linewidth]{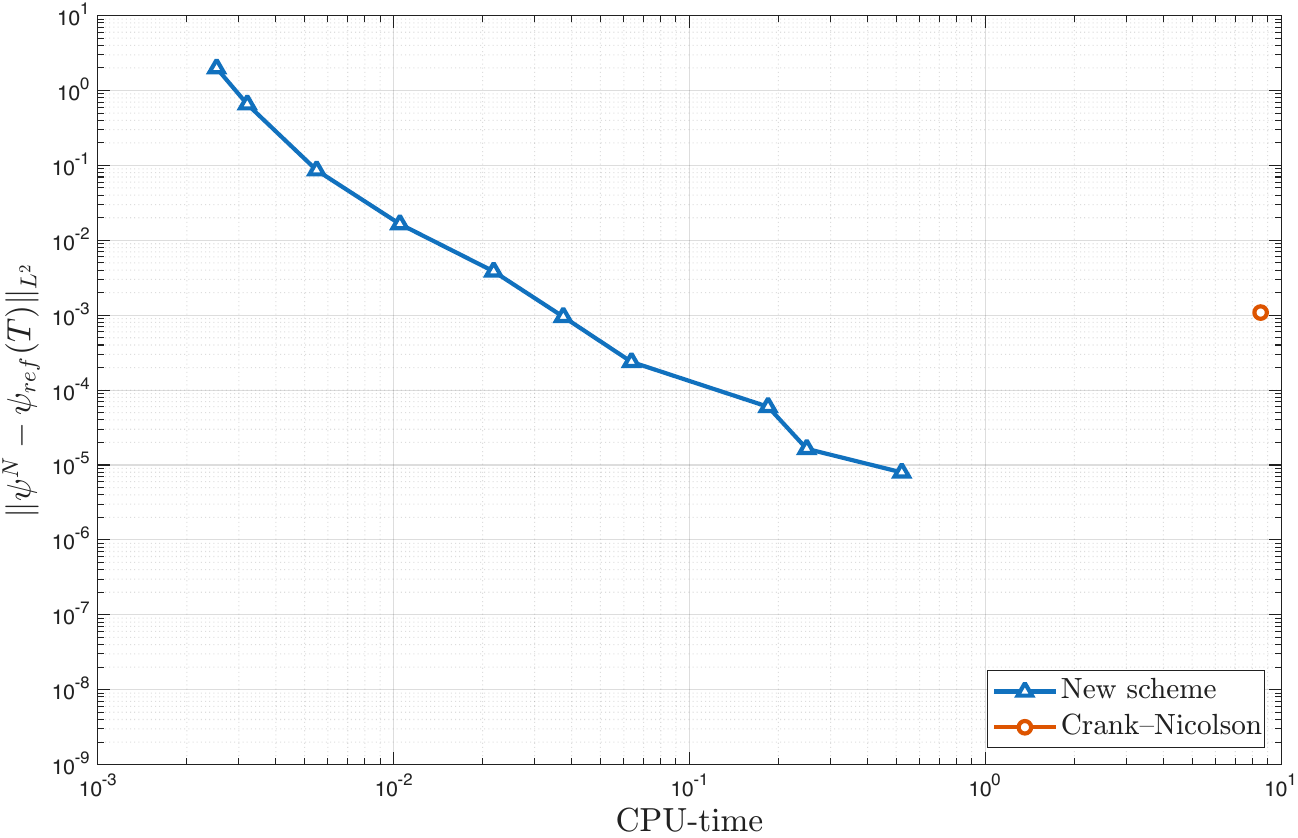}
    \caption{Error vs CPU time.}
    \label{fig:global_error}
    \end{subfigure}
    \caption{Global error of the numerical solution for $T=1.0, N=500$.}
    \label{fig:error_500_1-0}
\end{figure}

\section{Conclusions}
In this work, we revisited Hermite spectral methods for nonlinear Schr\"odinger equations posed on unbounded domains and demonstrated that their usefulness is not restricted to problems with a harmonic trap. The key point facilitating this is that, although Hermite functions diagonalise the harmonic oscillator rather than the Laplacian, the free Schr\"odinger flow remains stable in the weighted Sobolev spaces $\Sigma^k$ that naturally arise in Hermite expansions. This provides a rigorous foundation for Hermite-based time-splitting methods for dispersive equations on unbounded domains without imposing artificial periodicity or domain truncation.

We presented \revisions{four examples showing} applications of this approach: a Hermite-spectral splitting method for the nonlinear Schr\"odinger equation \revisions{in two and three spatial dimensions,} and, based on the novel use of a gauge transform, an unconditionally stable method for the derivative nonlinear Schr\"odinger equation. \revisions{In addition,} we \revisions{provided} a fully discrete convergence analysis for two use-cases. Overall, the results of this work suggest that Hermite spectral methods provide a robust and efficient alternative for nonlinear dispersive equations on unbounded domains, even in the absence of harmonic potentials.

\revisions{We emphasise that we do not claim Hermite spectral methods to be universally superior to Fourier spectral methods: per degree of freedom, the FFT remains cheaper than the Hermite transform, and for solutions that stay well-localised in a fixed region a Fourier method on a suitably chosen box is often the faster option. Rather, Hermite methods offer a reliable alternative that eliminates the need to choose a truncation window and faithfully represents the solution on the whole space, which (as our numerical results demonstrate) can lead to more efficient methods, including in the absence of harmonic potentials. The central contribution of this work is the theoretical guarantee that such methods are stable, and hence convergent, even in the absence of a harmonic trap.}

\section*{Acknowledgements}
The authors would like to thank Prof.\ Arieh Iserles (University of Cambridge) for several exciting discussions on Hermite expansions and insightful comments on a draft of this manuscript. The authors gratefully acknowledge funding from the European Research Council (ERC) under the European Union's Horizon 2020 research and innovation programme (grant agreement No.\ 850941). GM gratefully acknowledges support through a Scientific High Level Visiting Fellowship funded by the Higher Education, Research and Innovation Department of the Embassy of France in the UK.

\bibliographystyle{siam}     
\bibliography{biblio}
\addcontentsline{toc}{section}{References}

\begin{appendix}   \section{Properties of $\Sigma^k$}\label{app:properties_Sigma_k}
    \subsection{Algebra property}
    $\Sigma^k$ is an algebra with its associated norm:
    \begin{lemma}
        For $k\in\mathbb{N}, k> d/2$ there is a constant $C>0$ such that for any $u\in\Sigma^0,v\in\Sigma^k$ we have
        \begin{align*}
            \|uv\|_{\Sigma^0}\leq C\|u\|_{\Sigma^0}\|v\|_{\Sigma^k}.
        \end{align*}
        Moreover, there is a constant $C>0$ such that for any $u,v\in \Sigma^k$ we have
        \begin{align*}
            \|uv\|_{\Sigma^k}\leq C\|u\|_{\Sigma^k}\|v\|_{\Sigma^k}.
        \end{align*}
    \end{lemma}
\begin{proof}
   {{The $\||x|^kuv\|_{L^2}$-part of $\|uv\|_{\Sigma^k}$ is controlled by the Cauchy-Schwarz inequality. The remaining $\|uv\|_{H^k}$-part is controlled as stated because when $k>\frac d2$ we have that $H^k(\mathbb R^d)$ is an algebra and we also have that $H^k(\mathbb R^d)\subset L^\infty(\mathbb R^d)$.}}
\end{proof}
\section{Computing with Hermite expansions}\label{sec:computing_with_hermite_expansions}
\subsection{Stable Hermite transforms}\label{app:hermite_transform}
A central ingredient for the practical implementation of our numerical methods is the efficient transform between function values and Hermite coefficients. For the approximation of Hermite coefficients of a function $f(x)$ we use Gauss--Hermite quadrature: let
$\mathbf{x}=(x_0,\dots,x_{M-1})$ be the Gauss--Hermite quadrature
nodes on $\R$ of degree $M$ and $w_k, k=0,\dots M$ be the corresponding quadrature weights. Then
\begin{align*}
	\alpha_m:=\int_{\R} f(x) \mathcal{H}_m(x) dx=\int_{\R} f(x) \mathrm{H}_m(x)e^{\frac{x^2}{2}} e^{-x^2} dx \approx \sum_{k=0}^{M-1}w_kf(x_k)\mathrm{H}_m(x_k)e^{\frac{x_k^2}{2}}.
\end{align*}
 In other words the transformation \textit{from function values $f(x_m)$ to Hermite coefficients $\alpha_m$} can be expressed as the following matrix multiplication
\begin{align*}
	\bm{\alpha}=\mathsf{T}f(\mathbf{x}),
\end{align*}
where $\mathsf{T}_{km}=w_m\mathrm{H}_k(x_m)\exp(x_m^2/2).$ The entries
of $\mathsf{T}$ grow super-exponentially in the number of Hermite
modes $M$, and as a result both the assembly of the matrix and the
computation with this matrix become unstable even for moderate values
of $M$. This leads to numerical instabilities when computing with the Hermite basis and requires additional algorithmic insights to overcome. A stable method for assembling $\mathsf{T}$ was proposed by \cite[Appendix]{bunck09} and, more recently, an improved efficient way of computing with $\mathsf{T}$ was presented in \cite{maierhoferwebb25} (see also the Julia implementations in
\cite{QuantumTimeSteppers,FastGaussQuadrature,FastTransforms}). In all of our computations we use the algorithm presented in \cite{maierhoferwebb25}.
    \subsection{Differentiation}
The following identities are standard (cf. \cite[Appendix]{carles2025scattering}) but in the interest of completeness we include them here. We write the following in the 1D case, but clearly this can be extended to arbitrary $d\geq 1$ in tensor product form. Let us denote the Hermite coefficients of a function $f$ by $\alpha_m$, i.e.\newpage\ \vspace{-1.5cm}\ \\
\begin{align*}
	f(x)=\sum_{\substack{m\in\N^d\\0\le m\le M-1}} \alpha_m\mathcal{H}_m(x).
\end{align*}
To compute the derivative of a
Hermite expansion, we recall that the normalised Hermite polynomials are
of the form 
\begin{align*}
\mathrm{H}_{m}=\frac{1}{\sqrt{2^m\sqrt{\pi}m!}} \widetilde{\mathrm{H}}_m,
\end{align*}
where $\widetilde{\mathrm{H}}_m$ satisfies the identity
\begin{equation*}
  \frac{d}{dx}\left(\widetilde{\mathrm{H}}_{m}e^{-x^2/2}\right)=
  x\widetilde{\mathrm{H}}_{m}
  e^{-x^2/2}-\widetilde{\mathrm{H}}_{m+1}e^{-x^2/2},\quad  
  m\ge 0,
\end{equation*}
cf. \cite[18.9.26]{NIST:DLMF}. This leads to the identity
\begin{align}\label{eqn:differentiation_operator}
	f'(x)=\sum_{m\ge 0}\beta_m \mathrm{H}_m e^{-x^2/2},
\end{align}
where $\beta_m=\gamma_m-\sqrt{2(m+1)}\alpha_{m+1}$ and $\gamma_m$ are the Hermite coefficients of $xf(x)$, i.e.
\begin{align*}
	xf(x)=\sum_{m\ge 0}\gamma_m \mathrm{H}_m e^{-x^2/2}.
\end{align*}
Throughout this paper, we denote the differentiation operator by $\mathsf{A}_{\partial}$ such that
\begin{align}\label{eqn:differentiation_operator}
    \bm{\beta}=\mathsf{A}_{\partial}\bm{\alpha}.
\end{align}

Similarly, the action of the Laplacian can most easily be captured using the eigenfunction properties of $\mathcal{H}_m$. Indeed, by \eqref{eq:eigenfunction_property} we have
\begin{align*}
    \Delta f(x)=-(-\Delta+|x|^2)f(x)+|x|^2f(x)=\sum_{\substack{m\in\N^d\\0\le m\le M-1}} (-\lambda_m)\alpha_m\mathcal{H}_m(x)+|x|^2f(x),
\end{align*}
where the eigenvalues $\lambda_m$ are as in \eqref{eqn:details_sigma_norm_in_hermite_basis}. Thus we have, writing $\delta_m, m=0,\dots, M$ for the Hermite coefficients of $\Delta f$:
\begin{align}\label{eqn:Hermite_coeffs_of_laplacian}
    \bm{\delta}=\mathsf{A}_{\Delta}\bm{\alpha}:=-\mathsf{D}_{\lambda}\bm{\alpha}+\mathsf{T}\mathsf{D}_{|x|^2}\mathsf{T}^{-1}\bm{\alpha}
\end{align}
where $\mathsf{D}_{\lambda},\mathsf{D}_{|x|^2}$ are diagonal matrices with non-zero entries $\lambda_0,\dots, \lambda_{M-1}$ and $|x_0|^2,\dots,|x_{M-1}|^2$ respectively.
\subsection{Integration}\label{app:stable_integration}
In this section we describe how to stably compute
\begin{align*}
F(x):=\int_{-\infty}^x |f(y)|^2\,dy    
\end{align*}
for $x\in\mathbb{R}$ from $f(x)=\sum_{0\le m\le M-1} \alpha_m{H}_m(x)$. The difficulty is that, even if $f$ is
well-resolved in the Hermite basis, the primitive $F$
does not decay as $x\to+\infty$ but instead satisfies
$F(x)\to \|f\|_{L^2}^2$. Since rapid decay of Hermite coefficients is tied to
decay of the underlying function at $\pm\infty$, a direct computation of $F$
by inverting the differentiation operator \eqref{eqn:differentiation_operator} at the coefficient level is
typically ill-conditioned and leads to a loss
of accuracy in practice.

The way to overcome this is to introduce a correction term that stabilises this inversion. To restore decay, we subtract a smooth step function capturing the asymptotic constant.
Let
\begin{align}\label{eqn:Phi_def}
\Phi(x):=\int_{-\infty}^x \frac{1}{\sqrt{2\pi}}e^{-y^2/2}\,dy
= \frac12\Bigl(1+\mathrm{erf}\bigl(x/\sqrt{2}\bigr)\Bigr).
\end{align}
Then $\Phi(x)\to 0$ as $x\to-\infty$ and $\Phi(x)\to 1$ as $x\to+\infty$.
We define the {corrected primitive}
\begin{align}\label{eqn:corrected_primitive}
\widetilde{F}(x):=F(x)-\Phi(x)\,\|f\|_{L^2}^2.
\end{align}
By construction, $\widetilde{F}(x)\to 0$ as $x\to\pm\infty$, so
$\widetilde{F}$ can be represented accurately with a moderate number of
Hermite modes. Differentiating \eqref{eqn:corrected_primitive} yields
\begin{align}\label{eqn:corrected_rhs}
\widetilde{F}'(x)=|f(x)|^2-\Phi'(x)\,\|f\|_{L^2}^2
=|f(x)|^2-\frac{1}{\sqrt{2\pi}}e^{-x^2/2}\,\|f\|_{L^2}^2.
\end{align}
Since ${H}_0(x)=\pi^{-1/4}e^{-x^2/2}$, we have
$\frac{1}{\sqrt{2\pi}}e^{-x^2/2}
=
\frac{1}{\sqrt{2\sqrt{\pi}}}\,H_0(x)$, and therefore the correction term in coefficient space affects only the
zeroth mode.

\paragraph{Discrete algorithm.}
Let $\mathbf{g}:=|f(\mathbf{x})|^2$ be the integrand sampled at the Gauss--Hermite quadrature nodes $\mathbf{x}$, and let
$\bm{\rho}$ denote its Hermite coefficients computed via the stable transform described in Appendix~\ref{app:hermite_transform} (an alternative method is presented in \cite{LEIBON2008211}). We approximate the $L^2$-mass by Gauss--Hermite quadrature,
\[
\|f\|_{L^2}^2 \approx \sum_{k=0}^{M-1} w_k |f(x_k)|^2.
\]
Let $\mathbf{e}_0=(1,0,\dots,0)^\top\in\R^M$. We then form the corrected right
hand side in coefficient space as
\begin{align}\label{eqn:rhs_coeffs_corrected}
\bm{\rho}_{\mathrm{corr}}
:=
\bm{\rho}
-
\frac{\|f\|_{L^2}^2}{\sqrt{2\sqrt{\pi}}}\,\mathbf{e}_0.
\end{align}
Next, we compute the Hermite coefficients $\bm{\eta}$ of $\widetilde{F}$ by
solving the linear system
\begin{align}\label{eqn:solve_for_integral}
\mathsf{A}_{\partial}\bm{\eta}=\bm{\rho}_{\mathrm{corr}},
\end{align}
where $\mathsf{A}_{\partial}$ is the differentiation matrix introduced in \eqref{eqn:differentiation_operator}. Finally, we evaluate $\widetilde{F}$ at the quadrature nodes via the inverse
transform $\widetilde{F}(x)=\sum_{m=0}^{M-1}\eta_m H_m(x)$ and recover $F$ by adding back the smooth step function:
\begin{align}\label{eqn:recover_F}
F(\mathbf{x})
=
\widetilde{F}(\mathbf{x})
+
\Phi(\mathbf{x})\,\|f\|_{L^2}^2.
\end{align}
This procedure yields a stable approximation of $F$ at essentially the same
computational cost as the na\"ive inversion approach, while avoiding the loss of
resolution associated with the non-decaying asymptotics of $F$.

\paragraph{Computing the gauge factor $E$.}
In the main text we require
\[
E(x)=\exp\!\left(i\delta\int_{-\infty}^x |f(y)|^2\,dy\right).
\]
Using the values $F(\mathbf{x})$ obtained from \eqref{eqn:recover_F}, we set
\begin{align}\label{eqn:E_discrete}
E(\mathbf{x})=\exp\!\bigl(i\delta\,F(\mathbf{x})\bigr),
\end{align}
which is then used in the R-transform and its inverse.

\section{Proof of Lemma~\ref{lem:approx_of_u_0_v_0}}\label{app:proof_lemma_modified_initial_conditions}
\begin{proof}[Proof of Lemma~\ref{lem:approx_of_u_0_v_0}]
    Here we will show only the estimate \eqref{eqn:estimate_u_variable_initial_condition} since \eqref{eqn:estimate_v_variable_initial_condition} follows similarly. To prove \eqref{eqn:estimate_u_variable_initial_condition} we note firstly that, using Cauchy--Schwarz and Lemma~\ref{lem:Hermite_interpolation},
\begin{align*}
\left|e^{2i\delta\int_{-\infty}^x \mathcal{Q}_M(|\psi_0(y)|^2) -|\psi_0(y)|^2dy}-1\right|&\leq2|\delta|\left|\int_{-\infty}^x \mathcal{Q}_M(|\psi_0(y)|^2) -|\psi_0(y)|^2dy\right|\\
&\leq2|\delta| \left|\int_{-\infty}^x (1+|y|)^{-\alpha}(1+|y|)^{\alpha}\mathcal{Q}_M(|\psi_0(y)|^2) -|\psi_0(y)|^2dy\right|\\
&\leq2|\delta| \underbrace{\left(\int_{-\infty}^x (1+|y|)^{-2\alpha}dy\right)^{\frac{1}{2}}}_{<\infty}\|\mathcal{Q}_M(|\psi_0|^2) -|\psi_0|^2\|_{\Sigma^{\alpha}}\\
&\leq C_2 M^{1/3+\alpha/2-k/2}\|\psi_0\|^2_{\Sigma^k},
\end{align*}
for any $k>\alpha>1/2$ and a constant $C_2>0$ depending on $\alpha,k,\delta$.{Thus we have $\|(E_{0,M}(x))^2-(E_{0}(x))^2\|_{L^\infty}\leq  C_2 M^{1/3+\alpha/2-k/2}\|\psi_0\|^2_{\Sigma^k}$. This allows us to estimate $\|\tilde{u}^0_M-u_M^0\|_{\Sigma^\ell}$ for $\ell\geq 0$ as follows:
\begin{align*}
    \|\tilde{u}^0_M-u_M^0\|_{\Sigma^\ell}&=\|\mathcal{Q}_{M}\!\left[\left((E_{0,M})^2-(E_{0})^2\right)\psi_0\right]\!\|_{\Sigma^\ell}\leq c\|\left((E_{0,M})^2-(E_{0})^2\right)\psi_0\|_{\Sigma^{\ell+1}},
\end{align*}
for a constant $c>0$ depending only on $\ell$. This estimate follows from Lemma~\ref{lem:Hermite_interpolation}. Thus it remains to estimate $\|\left((E_{0,M})^2-(E_{0})^2\right)\psi_0\|_{\Sigma^{\ell+1}}.$} {Using \eqref{eqn:expression_Sigma_k_regularity+decay} we can estimate the Sobolev part and the polynomial part of the $\Sigma^\ell$-norm separately. Firstly we have for $j\geq 0$
\begin{align}\nonumber
\|x^{j}\!\left((E_{0,M})^2-(E_{0})^2\right)\psi_0\|_{L^2}&\leq \|(E_{0,M}(x))^2-(E_{0}(x))^2\|_{L^\infty}\|\psi_0\|_{\Sigma^{j}}\\\label{eqn:L2estimate_twist}
&\leq C_2 M^{1/3+\alpha/2-k/2}\|\psi_0\|^2_{\Sigma^k}\|\psi_0\|_{\Sigma^{j}}.
\end{align}
For the Sobolev part of the $\Sigma^\ell$-norm we will proceed by induction to show for any function $\psi_0$, we have for any $k>\max\{j-1,\alpha\}>1/2$
\begin{align}\label{eqn:estimate_sobolev_part}
    \|\partial_x^j\left[\left((E_{0,M})^2-(E_{0})^2\right)\psi_0\right]\|_{L^2}\leq  M^{1/3+\max\{j-1,\alpha\}/2-k/2}F(\|\psi_0\|_{\Sigma^k}),
\end{align}
where $F$ is a polynomial with positive coefficients depending only on $k,j,\alpha,\delta$. The estimate holds for $j=0$ by \eqref{eqn:L2estimate_twist}. For the induction step we note
\begin{align*}
     \|\partial_x^{j+1}&\left[\left((E_{0,M})^2-(E_{0})^2\right)\psi_0\right]\|_{L^2}=\|\partial_x^{j}\left[\left((E_{0,M})^2-(E_{0})^2\right)\partial_x\psi_0\right]\|_{L^2}+\|\partial_x^{j}\left[\partial_x\!\left((E_{0,M})^2-(E_{0})^2\right)\psi_0\right]\|_{L^2}\\
     &\leq M^{1/3+\max\{j-1,\alpha\}/2-k/2}F(\|\psi_0\|_{\Sigma^{k+1}}) +|2\delta|\|\partial_x^{j}\left[\left(\mathcal{Q}_M(|\psi_0|^2)(E_{0,M})^2-|\psi_0|^2(E_{0})^2\right)\psi_0\right]\|_{L^2}\\
     &\leq  M^{1/3+\max\{j-1,\alpha\}/2-k/2}F(\|\psi_0\|_{\Sigma^{k+1}})  +|2\delta|\left(\|\partial_x^{j}\left[\left(\mathcal{Q}_M(|\psi_0|^2)(E_{0,M})^2-\mathcal{Q}_M(|\psi_0|^2)(E_{0})^2\right)\psi_0\right]\|_{L^2}\right)\\
     &\quad +|2\delta|\left(\|\partial_x^{j}\left[\left(\mathcal{Q}_M(|\psi_0|^2)(E_{0})^2-|\psi_0|^2(E_{0})^2\right)\psi_0\right]\|_{L^2}\right)\\
     &\leq  M^{1/3+\max\{j-1,\alpha\}/2-k/2}F(\|\psi_0\|_{\Sigma^{k+1}}) +|2\delta|\left(\|\partial_x^{j}\left[\left((E_{0,M})^2-(E_{0})^2\right)\mathcal{Q}_M(|\psi_0|^2)\psi_0\right]\|_{L^2}\right)\\
     &\quad  +|2\delta|\left(\|\partial_x^{j}\left[\left(\mathcal{Q}_M(|\psi_0|^2)-|\psi_0|^2\right)(E_{0}^2\psi_0)\right]\|_{L^2}\right)\\
     &\leq  M^{1/3+\max\{j-1,\alpha\}/2-k/2}F(\|\psi_0\|_{\Sigma^{k+1}}) +M^{1/3+\max\{j-1,\alpha\}/2-k/2}\tilde{F}(\|\psi_0\|_{\Sigma^{k+1}})\\
     &\quad + C\|\mathcal{Q}_M(|\psi_0|^2)-|\psi_0|^2\|_{\Sigma^{j+1}}\|E_{0}^2\psi_0\|_{\Sigma^{j+1}}\\
      &\leq  M^{1/3+\max\{j-1,\alpha\}/2-k/2}F(\|\psi_0\|_{\Sigma^{k+1}}) +M^{1/3+\max\{j-1,\alpha\}/2-k/2}\tilde{F}(\|\psi_0\|_{\Sigma^{k+1}})\\
     &\quad + CM^{1/3+j/2-k/2}\tilde{\tilde{F}}(\|\psi_0\|_{\Sigma^{k+1}})\\
     &\leq C M^{1/3+\max\{j,\alpha\}/2-k'/2}P(\|\psi_0\|_{\Sigma^{k'}}),
\end{align*}
where in the final step we wrote $k'=k+1$ and $F,\tilde{F},\tilde{\tilde{F}},P$ denote polynomials with positive coefficients that only depend on $k,j,\alpha,\delta$. In summary this implies that for $k'>\ell+1$
\begin{align*}
    \|\left((E_{0,M})^2-(E_{0})^2\right)\psi_0\|_{\Sigma^{\ell+1}}\leq M^{1/3+\max\{\ell,\alpha\}/2-k'/2}F(\|\psi_0\|_{\Sigma^{k'}}),
\end{align*}
and thus
\begin{align*}
    \|\tilde{u}_M^0-u_M^0\|_{\Sigma^\ell}\leq M^{1/3+\max\{\ell,\alpha\}/2-k/2}F(\|\psi_0\|_{\Sigma^k}).
\end{align*}
Similarly we find
\begin{align*}
    \|\tilde{v}_M^0-v_M^0\|_{\Sigma^\ell}\leq M^{1/3+\max\{\ell,\alpha\}/2-k/2}F(\|\psi_0\|_{\Sigma^{k+1}}).
\end{align*}}
\end{proof}
\end{appendix}
\end{document}